\documentclass[12pt,a4paper,reqno]{amsart}
\usepackage[dvips]{graphicx}
\usepackage{verbatim}
\usepackage{hyperref}
\usepackage{color}
\usepackage{soul,xcolor}
\usepackage{amssymb}
\usepackage{latexsym}
\usepackage{amsmath}
\usepackage{bbm}
\usepackage{amsthm}
\usepackage{amsrefs}

\newcommand{\R}{\mathbbm{R}}    
\providecommand{\ab}[1]{\vert #1\vert}		
\newcommand{\N}{\mathbbm{N}}    
\newcommand{\te}{\theta}
\newcommand{\la}{\lambda}
\newcommand{\eps}{\varepsilon}
\providecommand{\norma}[1]{\lVert #1 \rVert}          
\providecommand{\Norma}[1]{\biggl\lVert #1 \biggr\rVert}          
\providecommand{\ab}[1]{\vert #1\vert}           
\newcommand{\Ei}{\operatorname{Ei}}      
\newcommand{\Sh}{\mathcal{S}}	

\newcommand{\ddirac}[1]{%
  \,\boldsymbol{\delta}\!\pvector{#1}\!}
\newcommand{\pvector}[1]{\begin{pmatrix}#1\end{pmatrix}}

\DeclareFontFamily{U}{mathx}{\hyphenchar\font45}
\DeclareFontShape{U}{mathx}{m}{n}{
      <5> <6> <7> <8> <9> <10>
      <10.95> <12> <14.4> <17.28> <20.74> <24.88>
      mathx10
      }{}
\DeclareSymbolFont{mathx}{U}{mathx}{m}{n}
\DeclareFontSubstitution{U}{mathx}{m}{n}
\DeclareMathAccent{\widecheck}{0}{mathx}{"71}
\DeclareMathAccent{\wideparen}{0}{mathx}{"75}
\numberwithin{equation}{section}
\theoremstyle{plain}
\newtheorem{teo}{Theorem}[section]
\newtheorem{cor}[teo]{Corollary}
\newtheorem{lemma}[teo]{Lemma}
\newtheorem{prop}[teo]{Proposition}

\theoremstyle{definition}
\newtheorem{define}[teo]{Definition}

\newtheorem{remark}[teo]{Remark}

\begin{document}
\title[Nonexistence of extremals for the hyperboloid]{Nonexistence of extremals for the adjoint restriction inequality on the hyperboloid}
\author{Ren\'e Quilodr\'an}
\address{Department of Mathematics, University of California, Berkeley, CA 94720-3840, USA}
\email{rquilodr@math.berkeley.edu}
\thanks{Research supported in part by NSF grant DMS-0901569.}
\setstcolor{red}
\begin{abstract}
 We study the problem of existence of extremizers for the $L^2$ to $L^p$ adjoint Fourier restriction inequalities for the
hyperboloid in dimensions $3$ and $4$ in the case $p$ is an even integer. We use the method developed by Foschi in
\cite{Fo} to show that extremizers do not exist.
\end{abstract}
\maketitle

\section{Introduction}
For $d\geq 1$, let $\mathbbm H^{d}$ denote the hyperboloid in $\R^{d+1}$, $\mathbbm H^{d}=\{(y,\sqrt{1+\ab{y}^2}):y\in \R^{d}\}$,
equipped with the measure 
\[\sigma(y,y')=\delta\Bigl(y'-\sqrt{1+\ab{y}^2}\Bigr)\frac{dydy'}{\sqrt{1+\ab{y}^2}}\]
defined by duality as
 \[\int_{\mathbbm H^d} g(y,y')d\sigma(y,y')=\int_{\R^d}g\Bigl(y,\sqrt{1+\ab{y}^2}\Bigr)\frac{dy}{\sqrt{1+\ab{y}^2}},\]
for all $g\in C_0(\R^{d+1})$.

A function $f:\mathbbm H^d\to \R$ can be identified with a function from $\R^d$ to $\R$, and in what follows, we do so. We denote
the $L^p(\mathbbm H^d,\sigma)$-norm of a function $f$ by $\norma{f}_{L^p(\mathbbm H^d)},\,\norma{f}_{L^p(\sigma)}$ or
$\norma{f}_p$.

The extension or adjoint Fourier restriction operator for $\mathbbm H^d$ is given by 
\begin{equation}
 \label{fourier-extension-operator-hyperboloid}
 Tf(x,t)=\int_{\R^d} e^{ix\cdot y} e^{it\sqrt{1+\ab{y}^2}} f(y)(1+\ab{y}^2)^{-1/2}dy,
\end{equation}
where $(x,t)\in\R^{d}\times \R$ and $f\in\mathcal S(\R^d)$. With the Fourier transform in $\R^{d+1}$ defined to be $\hat g(\xi)=\int_{\R^{d+1}} e^{-ix\cdot \xi}g(x)dx$, we see that $Tf(x,t)=\widehat{f\sigma}(-x,-t)$.

It is known \cite{Str} that there exists $C_{d,p}<\infty$ such that for all $f\in L^2(\mathbbm H^d)$, the estimate for
$Tf$
\begin{equation}
\label{restriction-hyperboloid}
 \norma{Tf}_{L^p(\R^{d+1})}\leq C_{d,p}\norma{f}_{L^2(\mathbbm H^d)}
\end{equation}
holds provided that
\begin{equation}
\label{numerology}
 \begin{aligned}
  \frac{2(d+2)}{d}\leq &p\leq \frac{2(d+1)}{d-1} &&\text{ if } d>1,\\
  6\leq &p<\infty, &&\text{ if } d=1.
 \end{aligned}
\end{equation}

For $p$ satisfying \eqref{numerology}, we denote by $\mathbf{H}_{d,p}$ the best constant in \eqref{restriction-hyperboloid},
\[\mathbf{H}_{d,p}=\sup\limits_{0\neq f\in L^2(\mathbbm H^d)}\frac{\norma{Tf}_{L^p(\R^{d+1})}}{\norma{f}_{L^2(\mathbbm H^d)}}.\]

We also consider the two-sheeted hyperboloid 
\[\bar{\mathbbm H}^d=\{(y,y')\in\R^d\times\R:y'^2=1+\ab{y}^2\}.\]
and endow it with the measure $\bar\sigma=\sigma^++\sigma^-$, where
\begin{align*}
 \sigma^+(y,y')&=\sigma(y,y')=\delta\Bigr(y'-\sqrt{1+\ab{y}^2}\Bigl)\frac{dydy'}{\sqrt{1+\ab{y}^2}},\\
 \sigma^-(y,y')&=\delta\Bigr(y'+\sqrt{1+\ab{y}^2}\Bigl)\frac{dydy'}{\sqrt{1+\ab{y}^2}}.\\
\end{align*}
The corresponding adjoint Fourier restriction operator is $\bar Tf=\widehat{f\sigma^+}+\widehat{f\sigma^-}$.
If $(d,p)$ satisfies \eqref{numerology}, then
\begin{equation}
\label{restriction-two-sheeted-hyperboloid}
 \norma{\bar Tf}_{L^p(\R^{d+1})}\leq \bar{\mathbf{H}}_{d,p}\norma{f}_{L^2(\bar{\mathbbm H}^d)},
\end{equation}
where
\begin{equation}
 \label{best-constant-two-sheeted}
 \bar{\mathbf{H}}_{d,p}=\sup\limits_{0\neq f\in L^2(\bar{\mathbbm H}^d)}\frac{\norma{\bar
Tf}_{L^p(\R^{d+1})}}{\norma{f}_{L^2(\bar{\mathbbm H}^d)}}
\end{equation}
is finite.

\begin{define}
\label{definition-extremizer}
 An \textbf{extremizing sequence} for inequality \eqref{restriction-hyperboloid} is a sequence $\{f_n\}_{n\in\N}$ of functions
in
$L^2(\mathbbm H^d)$ satisfying $\norma{f_n}_{L^2(\mathbbm H^d)}\leq 1$ such that 
\[\norma{Tf_n}_{L^p(\R^{d+1})}\to \mathbf{H}_{d,p}\text{ as } n\to\infty.\]
 
 An \textbf{extremizer} for inequality \eqref{restriction-hyperboloid} is a function $f\neq 0$ which satisfies
$\norma{Tf}_{L^p(\R^{d+1})}=\mathbf{H}_{d,p}\norma{f}_{L^2(\mathbbm H^d)}$. These terms are defined analogously for inequality
\eqref{restriction-two-sheeted-hyperboloid}.
\end{define}

We are interested in the following pairs $(2,4),\,(2,6)$ and $(3,4)$ of $(d,p)$, which are the only cases for $d>1$ where $p$ is
an even integer. The main result of this paper is the following theorem.

\begin{teo}
\label{main-theorem}
 The values of the best constants are $\mathbf{H}_{2,4}=2^{3/4}\pi,\, \mathbf{H}_{2,6}=(2\pi)^{5/6}$ and
$\mathbf{H}_{3,4}=(2\pi)^{5/4}$. Moreover, extremizers for inequality \eqref{restriction-hyperboloid} do not exist in each of the
three cases of $(d,p)$.
 
 The best constants for the two-sheeted hyperboloid are $\bar{\mathbf{H}}_{2,4}=(3/2)^{1/4}\mathbf{H}_{2,4},\, 
$ and $\bar{\mathbf{H}}_{3,4}=(3/2)^{1/4}\mathbf{H}_{3,4}$, and
extremizers for inequality \eqref{restriction-two-sheeted-hyperboloid}  do not exist.
\end{teo}

When $(d,p)=(2,6)$ we only prove an inequality for $\bar{\mathbf{H}}_{2,6}$ as recorded in the next remark.

\begin{remark}\label{new-remark}
In Proposition \ref{two-sheeted-hyperboloid-p6} we show that for each $f\in L^2(\bar{\mathbbm H}^2)$, $f\neq 0$
\begin{equation*}
\norma{\bar T f}_{L^6(\R^{3})}^6\norma{f}_{L^2(\bar{\mathbbm H}^2)}^{-6}< \frac{25}{4}\mathbf{H}_{2,6}^6,
\end{equation*}
and therefore $\bar{\mathbf{H}}_{2,6}\leq (5/2)^{1/3}\mathbf{H}_{2,6}$. Moreover, a refinement of the argument shows that there is a strict inequality $\bar{\mathbf{H}}_{2,6}<(5/2)^{1/3}\mathbf{H}_{2,6}$.
\end{remark}

We normalize the Fourier transform in $\R^d$ as $\hat g(\xi)=\int_{\R^d} e^{-ix\cdot \xi}g(x)dx$. With this normalization, the
convolution and $L^2(\R^d)$ norm satisfy
\[\widehat{f* g}=\hat f\, \hat g \,\text{ and }\,\norma{\hat f}_{L^2(\R^d)}=(2\pi)^{d/2}\norma{f}_{L^2(\R^d)},\]
respectively. If $p=2k$ is an even integer, we can write \eqref{restriction-hyperboloid} in ``convolution form''
\begin{equation}
\label{convolution-form-hyperboloid}
\begin{aligned}
 \norma{Tf}_{L^{2k}
(\R^{d+1})}^{k}&=\norma{(Tf)^{k}}_{L^2(\R^{d+1})}=\norma{(\widehat{f\sigma})^{k}}_{L^2(\R^{d+1})}=\norma{
(f\sigma*\dots*f\sigma)\hat{}\,}_{L^2(\R^{d+1})}\\
 &=(2\pi)^{(d+1)/2}\norma{f\sigma*\dots*f\sigma}_{L^2(\R^{d+1})},
\end{aligned}
\end{equation}
where $f\sigma*\dots*f\sigma$ denotes the $k$-fold convolution of $f\sigma$ with itself. Therefore, for $p$ an even integer,
\eqref{restriction-hyperboloid} is equivalent to
\[\norma{f\sigma*\dots*f\sigma}_{L^2(\R^{d+1})}^{1/k}\leq (2\pi)^{-(d+1)/2k}C_{d,2k}\norma{f}_{L^2(\mathbbm H^d)}\, \text{ for
all }f\in \mathcal S(\R^{d+1}).\]

For reference, we write here the best constants in convolution form:
\begin{align}
 \sup\limits_{0\neq f\in L^2(\mathbbm H^2)} \norma{f\sigma*f\sigma}_{L^2(\R^3)}^{1/2}\norma{f}_{L^2(\mathbbm
H^2)}^{-1}&=\pi^{1/4},\\
 \sup\limits_{0\neq f\in L^2(\mathbbm H^2)} \norma{f\sigma*f\sigma*f\sigma}_{L^2(\R^3)}^{1/3}\norma{f}_{L^2(\mathbbm
H^2)}^{-1}&=(2\pi)^{1/3},\\
 \sup\limits_{0\neq f\in L^2(\mathbbm H^3)} \norma{f\sigma*f\sigma}_{L^2(\R^4)}^{1/2}\norma{f}_{L^2(\mathbbm
H^3)}^{-1}&=(2\pi)^{1/4}.
\end{align}

It would be interesting to analyze the case $d=1$ for even integers greater than or equal to $6$. Our argument relies on the
explicit
computation of the $n$-fold convolution of the measure $\sigma$ with itself, and this seems to be computationally complicated if
$n\geq 3$. 

Interpolation shows that $\mathbf{H}_{2,p}\leq \mathbf{H}_{2,4}^\te \mathbf{H}_{2,6}^{1-\te}$ for $d=2$ and $p\in[4,6]$, 
where $\frac{1}{p}=\frac{\te}{4}+\frac{1-\te}{6}$. We do not know whether extremizers exist for $p\in (4,6)$, as our method only
applies when $p$ is an even integer.

We consider, for $s>0$, the hyperboloid $\mathbbm H_s^d=\{(y,\sqrt{s^2+\ab{y}^2}):y\in\R^d\}$, equipped with the measure
\begin{equation}
 \label{measure-sigma-s}
 \sigma_s(y,y')=\delta\Bigl(y'-\sqrt{s^2+\ab{y}^2}\Bigr)\frac{dydy'}{\sqrt{s^2+\ab{y}^2}}.
\end{equation}
As we mention in Section \ref{section-on-lorentz}, this measure is natural, since up to multiplication by scalar, it is the only
Lorentz invariant measure on $\mathbbm H^d_s$. Let $T_sf(x,t)=\widehat{f\sigma_s}(x,t)$. For $(d,p)$ satisfying
\eqref{numerology},
\begin{equation}
 \label{restriction-hyperboloid-s}
 \norma{T_sf}_{L^p(\R^{d+1})}\leq \mathbf{H}_{d,p,s}\norma{f}_{L^2(\mathbbm H^d_s)},
\end{equation}
where
\begin{equation}
 \label{hdps}
 \mathbf{H}_{d,p,s}=\sup\limits_{0\neq f\in L^2(\mathbbm H^d_s)}\frac{\norma{T_sf}_{L^p(\R^{d+1})}}{\norma{f}_{L^2(\mathbbm
H^d_s)}}
\end{equation}
is a finite constant.

As shown in Appendix 1, simple scaling relates $\mathbf{H}_{d,p,s}$ and $\mathbf{H}_{d,p}$ by
\begin{equation}
 \label{scaling}
 \mathbf{H}_{d,p,s}=s^{(d-1)/2-(d+1)/p}\mathbf{H}_{d,p}.
\end{equation}
Moreover, $\{f_n\}_{n\in\N}$ is a extremizing sequence for inequality \eqref{restriction-hyperboloid} if and only if
the sequence $\{s^{-1/2}f_n(s^{-1}\cdot)\}_{n\in\N}$ is extremizing for inequality \eqref{restriction-hyperboloid-s}.
Thus, for the problem of extremizers and properties of extremizing sequences, it is enough to study the case $s=1$.

For each $\rho\in(0,\infty)$, we consider the truncated hyperboloid 
\[\mathbbm H^d_{s,\rho}=\Bigl\{\Bigl(y,\sqrt{s^2+\ab{y}^2}\Bigr):y\in\R^d,\, \ab{y}\leq \rho\Bigr\},\]
endowed with the measure which is the restriction of $\sigma_s$ to $\mathbbm H^d_{s,\rho}$. For $f\in L^2(\mathbbm
H^d_{s,\rho})$, let $T_{s,\rho} f=T_sf$ denote the corresponding adjoint Fourier restriction operator. Since $\norma{T_{s,\rho}
f}_{L^\infty(\R^d)}\leq C\norma{f}_{L^2(\mathbbm H^d_{s,\rho})}$, it follows that for $d\geq 1$,
\begin{equation}
 \label{restricted-hyperboloid}
 \norma{T_{s,\rho} f}_{L^p(\R^{d+1})}\leq C\norma{f}_{L^2(\mathbbm H^d_{s,\rho})}
\end{equation}
for $p\geq 2(d+2)/d$ and some constant $C=C(d,p,s,\rho)<\infty$.

The main theorem of Fanelli, Vega and Visciglia in \cite{FVV}*{Theorem 1.1} implies that if $d\geq 1$ and $p>2(d+2)/d$,
complex-valued extremizers
for \eqref{restricted-hyperboloid} exist. There exist nonnegative extremizers if $p$ is an even integer, as can be seen from the
equivalent ``convolution form'' of \eqref{restricted-hyperboloid}. This shows that for $(d,p)=(2,6)$ and $(d,p)=(3,4)$, there
exist
extremizers for \eqref{restricted-hyperboloid}. The case $(d,p)=(2,4)$ does not follow from the result in \cite{FVV}, since it is
the endpoint. In Proposition \ref{result_truncated_hyperboloid}, we prove that in this case,
extremizers do not exist and that the best constant in \eqref{restricted-hyperboloid} is independent of $\rho$
and equals the best constant for the full hyperboloid $\mathbbm H^2_{s}$.
\section{Some related results}

In this section, we discuss the results in \cite{FVV2} and their connection to the case of the adjoint Fourier restriction
inequalities for the hyperboloid analyzed in this paper. 

For $r\in\R$, the \textbf{(nonhomogeneous) Sobolev space} $H^r(\R^d)$ consists of tempered distributions $g$ over $\R^d$ such
that $\hat g\in L^2_{loc}(\R^d)$ and the norm 
\[\norma{g}_{H^r(\R^d)}^2:=\int_{\R^d}\ab{\hat g(y)}^2(1+\ab{y}^2)^rdy\]
is finite.
The \textbf{homogeneous Sobolev space} $\dot H^r(\R^d)$ is the space of tempered distributions $g$ over $\R^d$ such that $\hat
g\in
L^1_{loc}(\R^d)$ and the norm 
\[\norma{g}_{\dot H^r(\R^d)}^2:=\int_{\R^d}\ab{\hat g(y)}^2\ab{y}^{2r}dy\]
is finite. Note that the
$\dot H^r(\R^d)$-norm satisfies the scaling property $\norma{g(\la\cdot)}_{\dot H^r(\R^d)}=\la^{r-d/2}\norma{g}_{\dot H^r(\R^d)}$.

Let us introduce the notation used in \cite{FVV2}. For a function $h:\R^d\to\R$, the operator
$e^{it h(D)}$ is defined by
\begin{equation}
 \label{kg-propagator}
 e^{it h(D)}g(x)=\frac{1}{(2\pi)^{d}}\int_{\R^d}e^{ix\cdot y}e^{it h(y)}\hat g(y)dy\,\text{ for } g\in\Sh(\R^d),
\end{equation}
 and for a function $\eta:\R\to\R$ we define $e^{it \eta(\sqrt{-\Delta})}=e^{it \eta(\ab{D})}$. 

Note that
$e^{it\sqrt{-\Delta+s^2}}g(x)=(2\pi)^{-d}Tf(x,t)$, where $\hat
g(y)=f(y)(s^2+\ab{y}^2)^{-1/2}$; therefore, \eqref{restriction-hyperboloid} is equivalent to the estimate
\begin{equation}
 \label{estimate-propagator}
 \norma{e^{it\sqrt{-\Delta+s^2}}g}_{L_{t,x}^p(\R^{d+1})}\leq C_{d,p,s}\norma{g}_{H^{\frac{1}{2}}(\R^d)}
\end{equation}
for a constant $C_{d,p,s}<\infty$ and $p$ as in \eqref{numerology}. For $s>0$, the operator $e^{it\sqrt{-\Delta+s^2}}$ satisfies
more
general mixed-norm Strichartz estimates, namely,
\begin{equation}
 \label{mixed-norm-strichartz}
 \norma{e^{it\sqrt{-\Delta+s^2}}g}_{L^p_tL^q_x(\R^{d+1})}\leq C\norma{g}_{H^{\frac{1}{p}-\frac{1}{q}+\frac{1}{2}}(\R^d)},
\end{equation}
where $p\in[2,\infty]$, $q\in[2,2d/(d-2)]$ ($q\in[2,\infty]$ if $d=1,2$), and
\[\frac{2}{p}+\frac{d-1+\te}{q}=\frac{d-1+\te}{2},\; (p,q)\neq (2,\infty)\]
Here, $\te\in[0,1]$. We refer the reader to \cite{KO} and the references therein for these
estimates.

Using \eqref{mixed-norm-strichartz}, the Sobolev Embedding Theorem, and interpolation, we obtain that for
$d\geq 2$,
\begin{equation}
 \label{hyperboloid-sobolev-norm}
 \norma{e^{it\sqrt{-\Delta+s^2}}g}_{L_{t,x}^{p}(\R^{d+1})}\leq
C\norma{g}_{H^{r}(\R^d)}
\end{equation}
for all $p$ and $r$ satisfying
\begin{equation}
 \label{condition-on-r-p}
 \frac{1}{2}\leq r<\frac{d}{2},\quad\frac{2(d+2)(d-1)}{d(d-2r)}\leq p\leq
\frac{2(d+1)}{d-2r}.
\end{equation}

An equivalent way to look at the adjoint Fourier restriction inequalities for the hyperboloid $\mathbbm H^d_{s}$ is through
Strichartz
estimates for the Klein-Gordon equation
\begin{equation}
\label{klein-gordon}
\begin{split}
 \partial_t^2 u&=\Delta u-s^2 u\quad \text{in}\;\R^{d+1}\\
 u(0,x)&=u_0(x),\quad \partial_tu(x,x)=u_1(x).
\end{split}
\end{equation}

Writing the solution of \eqref{klein-gordon} as 
\[u(t,\cdot)=\cos(t\sqrt{-\Delta+s^2})u_0(\cdot)+\frac{\sin(t\sqrt{-\Delta+s^2})}{\sqrt{-\Delta+s^2}}u_1(\cdot),\]
or equivalently as
\begin{equation*}
 \begin{aligned}
 u(t,\cdot)=\frac{1}{2}\biggl(e^{it\sqrt{-\Delta+s^2}}u_0(\cdot)&+\frac{1}{i}\frac{e^{it\sqrt{-\Delta+s^2}}}{\sqrt{-\Delta+s^2}}
u_1(\cdot)\biggr)\\
&\quad+
\frac{1}{2}\biggl(e^{-it\sqrt{-\Delta+s^2}}u_0(\cdot)-\frac{1}{i}\frac{e^{-it\sqrt{-\Delta+s^2}}}{\sqrt{-\Delta+s^2}}
u_1(\cdot)\biggr),
\end{aligned}
\end{equation*}
we see that \eqref{hyperboloid-sobolev-norm} is equivalent to the Strichartz estimate for $u$
\begin{equation}
 \label{stricartz-klein-gordon}
 \norma{u}_{L_{t,x}^p(\R^{d+1})}\leq
C\norma{(u_0,u_1)}_{H^{r}(\R^d)\times H^{r-1}(\R^d)},
\end{equation}
where $\norma{(u_0,u_1)}_{H^{r}(\R^d)\times
H^{r-1}(\R^d)}^2:=\norma{u_0}_{H^{r}(\R^d)}^2+\norma{u_1}_{H^{r-1}(\R^d)}^2$, $p$ and $r$ are as in
\eqref{condition-on-r-p}, and $C<\infty$ is a constant depending only on $d,p,r$ and $s$.

In the context of this paper, it is natural to ask whether inequalities \eqref{hyperboloid-sobolev-norm} and
\eqref{stricartz-klein-gordon} admit
extremizers $g\in H^r(\R^d)$ and $(u_0,u_1)\in H^{r}(\R^d)\times H^{r-1}(\R^d)$, respectively, and whether extremizing sequences
are precompact, after the possible application of symmetries. Here, extremizers and extremizing sequences are defined similarly
as for inequality \eqref{restriction-hyperboloid} in Definition \ref{definition-extremizer}.

In \cite{FVV2}, the existence of extremals and precompactness of extremizing sequences is studied for an inequality of the form
\begin{equation}
\label{FVV-inequality}
 \norma{e^{ith(D)}g}_{L^p_{t,x}(\R^{d+1})}\leq C\norma{g}_{\dot H^r(\R^d)},
\end{equation}
for operators $e^{ith(D)}$ that satisfy mixed-norm
estimates 
 \[\norma{e^{ith(D)}g}_{L^p_tL^q_x(\R^{d+1})}\leq C\norma{g}_{\dot H^{r}(\R^d)}\]
for some $0<r<d/2$ and $p$ and $q$ satisfying $2\leq p<q\leq \infty$ where the function $h(\xi)$ is homogenous of some degree
$k>0$,
meaning that $h(\la\xi)=\la^k h(\xi)$ for all $\la>0$ and $\xi\in\R^d$. 

The argument in \cite{FVV2} uses the
homogeneous
Sobolev spaces $\dot H^r(\R^d)$ and that the function $h(\xi)$ is
homogenous. Indeed, it is the scaling property of the $\dot
H^r(\R^d)$-norm and the homogeneity of the 
function $h$ that
imply that \eqref{FVV-inequality} is invariant under scaling, and therefore the
sequence defined
in \cite{FVV2}*{Equation 2.15} is still an extremizing sequence for the same inequality.

For the hyperboloid, the function $h(\xi)=\sqrt{s^2+\ab{\xi}^2}$ is not homogeneous if $s\neq 0$. Therefore, in this case, the
question
of existence
of extremizers in $H^r(\R^d)$, $1/2<r<d/2$, for inequality \eqref{hyperboloid-sobolev-norm} is not
answered in \cite{FVV2}, although information can be obtained from arguments therein, which we
record in Proposition
\ref{extremizing-non-endpoint}. We can contrast this situation with the case of the cone $\Gamma^d=\{(y,\ab{y}):y\in\R^d\}$ with
its dilation invariant measure $\sigma_0=\delta(y'-\ab{y})\ab{y}^{-1}dydy'$. This cone can be seen as the limiting case of the
hyperboloid $(\mathbbm H^d_s,\sigma_s)$ as $s\to 0$. 

Let $T_c$ denote the adjoint Fourier restriction operator on the cone $(\Gamma^d,\sigma_0)$:
\begin{equation}
 \label{fourier-extension-operator-cone}
 T_cf(x,t):=\int_{\R^d} e^{ix\cdot y} e^{it\ab{y}} f(y)\ab{y}^{-1}dy,\; \text{ for } f\in\Sh(\R^d).
\end{equation}

The operator $e^{it\sqrt{-\Delta}}$ is related to $T_c$ by $e^{it\sqrt{-\Delta}}g(x)=(2\pi)^{-d}T_cf(x,t)$,
where $\hat g(y)=f(y)\ab{y}^{-1}$. For $d\geq 2$, the operator $e^{it\sqrt{-\Delta}}$ satisfies
\begin{equation}
 \label{cone-sobolev-norm}
 \norma{e^{it\sqrt{-\Delta}}g}_{L^{\frac{2(d+1)}{d-2r}}_{t,x}(\R^{d+1})}\leq C\norma{g}_{\dot{H}^r(\R^d)},\;\frac{1}{2}\leq
r<\frac{d}{2}.
\end{equation}

The main result of \cite{FVV2}, Theorem 1.1, implies that for $d\geq 2$, extremizers exist for inequality
\eqref{cone-sobolev-norm} for every $1/2<r<d/2$ and, moreover, extremizing sequences are precompact after the
application of symmetries.

For the case $r=1/2$, the existence of extremizers was proved by Carneiro \cite{Ca} in the cases $d=2$ and $d=3$; 
he also found the exact form of the extremizers. The precompactness of extremizing sequences after the application of symmetries,
and
thus the existence
of extremizers, was proved in \cite{RQ} for $d=2$ and by Ramos \cite{Ramos} for $d\geq 2$.

The limiting case of the Klein-Gordon equation \eqref{klein-gordon} as $s\to0$ is the wave equation
\begin{equation}
\begin{split}
 \label{wave-equation}
 \partial_t^2 u&=\Delta u\quad \text{in}\;\R^{d+1},\\
 u(0,x)&=u_0(x),\, \partial_tu(x,x)=u_1(x).
\end{split}
\end{equation}
Its solution can be written as
\begin{equation*}
 u(t,\cdot)=\frac{1}{2}\biggl(e^{it\sqrt{-\Delta}}u_0(\cdot)+\frac{1}{i}\frac{e^{it\sqrt{-\Delta}}}{\sqrt{-\Delta}}
u_1(\cdot)\biggr)+
\frac{1}{2}\biggl(e^{-it\sqrt{-\Delta}}u_0(\cdot)-\frac{1}{i}\frac{e^{-it\sqrt{-\Delta}}}{\sqrt{-\Delta}}
u_1(\cdot)\biggr)
\end{equation*}
and satisfies, for $d\geq 2$ and $1/2\leq r<d/2$,
\begin{equation}
 \label{stricartz-wave-equation}
 \norma{u}_{L^{\frac{2(d+1)}{d-2r}}_{t,x}(\R^{d+1})}\leq
C\norma{(u_0,u_1)}_{\dot H^{r}(\R^d)\times \dot H^{r-1}(\R^d)}.
\end{equation}

Just as for the case of the adjoint Fourier restriction inequality for the cone, there are results concerning the existence of
extremizers for inequality
\eqref{stricartz-wave-equation}, $(u_0,u_1)\in \dot H^r(\R^d)\times\dot H^{r-1}(\R^d)$. Foschi \cite{Fo} studied the case
$r=1/2$ for $d=2$ and $d=3$, proved the existence of extremizers, and found their exact form. 
The existence of extremizers for \eqref{stricartz-wave-equation} when $d\geq 2$ and $1/2< r<d/2$ was proved in
\cite{FVV2}, while the case $d\geq 2$ and
$r=1/2$ was proved by Ramos \cite{Ramos}. See also the discussion at the end of \cite{FVV2}*{Example 1.4} for
complementary
results.

We note that the argument in \cite{FVV2} does not apply to inequality \eqref{stricartz-klein-gordon} for the same reasons
stated before for inequality \eqref{hyperboloid-sobolev-norm}.

Let us return to inequality \eqref{hyperboloid-sobolev-norm}, where we consider the nonendpoint case, that is, the case $p$ and
$r$ satisfy 
\begin{equation}
 \label{numerology-r-p}
  \frac{1}{2}\leq r<\frac{d}{2},\,\frac{2(d+2)(d-1)}{d(d-2r)}< p\leq \frac{2(d+1)}{d-2r},
\end{equation}
that is, \eqref{condition-on-r-p} with the endpoint $p=2(d+2)(d-1)/d(d-2r)$ removed.

In the next proposition, we show that the only obstruction to the convergence of extremizing sequences for inequality 
\eqref{hyperboloid-sobolev-norm}, after the applications of symmetries, is ``concentration at
infinity'' of the Fourier transform.

\begin{prop}
 \label{extremizing-non-endpoint}
 Suppose that $p$ and $r$ satisfy \eqref{numerology-r-p}. Let $\{g_n\}_{n\in\N}$ be an extremizing sequence for inequality
\eqref{hyperboloid-sobolev-norm}. Then one of the following two possibilities holds.
\begin{itemize}
 \item [(i)] For all $R\in(0,\infty),\lim_{n\to\infty}\int_{\ab{y}\leq R}\ab{\hat g_n(y)}^2(1+\ab{y}^2)^{r}dy=0$.

  \item [(ii)] There exist a subsequence $\{g_{n_k}\}_{k\in\N}$ and a sequence $\{(y_k,t_k)\}_{k\in\N}\subset\R^d\times\R$ such
that $\{e^{it_k\sqrt{-\Delta+s^2}}g_{n_k}(y-y_k)\}_{k\in\N}$ converges in $H^r(\R^d)$.
\end{itemize}
 Moreover, if (i) holds, then there exist a subsequence $\{g_{n_k}\}_{k\in\N}$, a sequence of positive real numbers
$\{\la_k\}_{k\in\N}$, $\la_k\to 0$ as $k\to\infty$, a sequence $\{(y_k,t_k)\}_{k\in\N}\subset\R^d\times \R$, and $0\neq v\in \dot
H^r(\R^d)$ such that
\begin{equation}
 \label{weak-convergence-homogeneous}
 \la_k^{d/2-r}e^{it_{k}\sqrt{-\Delta+s^2}}g_{n_k}(\la_k(y-y_k))\rightharpoonup v \text{ as }k\to\infty
\end{equation}
weakly in the homogeneous Sobolev space $\dot H^r(\R^d)$.
\end{prop}

In the dual formulation, in ``physical space'' instead of ``frequency space'', that is, via the equality $\hat
g(y)=f(y)(s^2+\ab{y}^2)^{-1/2}$, inequality \eqref{hyperboloid-sobolev-norm} becomes the weighted estimate
\begin{equation}
 \label{hyperboloid-sobolev-equivalent}
\norma{T_s f}_{L^p(\R^d)}\leq C\norma{f}_{L^2(\mu_s)},
\end{equation}
where the measure $\mu_s(y,y')=(1+\ab{y}^2)^r(s^2+\ab{y}^2)^{-1/2}\sigma_s(y,y')$ is supported on $\mathbbm{H}_{s}^d$. 

The two possibilities in the previous proposition, when written for \eqref{hyperboloid-sobolev-equivalent}, are as follows.
\begin{itemize}
 \item [(i')] The sequence $\{f_n\}_{n\in\N}$ concentrates at spatial infinity, that is, for all $R\in(0,\infty)$,
 \[\lim\limits_{n\to\infty}\int_{\ab{y}\leq R}\ab{f_n(y)}^2(s^2+\ab{y}^2)^{r-1}dy=0.\]
  \item [(ii')] There exist a subsequence $\{f_{n_k}\}_{k\in\N}$ and a sequence $\{(y_k,t_k)\}_{k\in\N}\subset\R^d\times\R$ such
that $\{e^{iy\cdot y_k}e^{it_k\sqrt{s^2+\ab{y}^2}}f_{n_k}(y)\}_{k\in\N}$ converges in $L^2(\mu_s)$.
\end{itemize}
For a set $A\subseteq \R^d$ we denote $\chi_A$ the characteristic function of the set $A$.

\begin{proof}[Sketch of the proof of Proposition \ref{extremizing-non-endpoint}]
 If condition (i) is not satisfied, then there exist $R\in(0,\infty)$ and a subsequence of $\{g_n\}_{n\in\N}$, which we also call
$\{g_n\}_{n\in\N}$, satisfying
\begin{equation}
 \label{reverse-i}
 \inf\limits_{n\in\N}\int_{\ab{y}\leq R}\ab{\hat g_n(y)}^2(1+\ab{y}^2)^rdy=:\eps^2>0.
\end{equation}

We define $g_{n,1}$ and $g_{n,2}$ by their Fourier transforms, $\hat{g}_{n,1}(y)=\hat g_{n}(y)\chi_{\{\ab{y}\leq R\}}$, $\hat
g_{n,2}(y)=\hat g_{n}(y)\chi_{\{\ab{y}>R\}}$. Then $g_n=g_{n,1}+g_{n,2}$; and for all large enough $n$, we have
$\norma{g_{n,1}}_{H^r(\R^d)}\geq \eps/2$
and $\norma{e^{it\sqrt{-\Delta+s^2}}g_{n,1}}_{L^p(\R^{d+1})}\geq c>0$ for a certain constant $c$ independent of $n$.

Under the assumptions on $p$ and $r$, we can apply the
``first step`` in
the proof of \cite{FVV}*{Theorem 1.1} to the sequence $\{g_{n,1}\}_{n\in\N}$ to show that there
exist a
subsequence, which we also call $\{g_{n,1}\}_{n\in\N}$, and a sequence $\{(y_n,t_n)\}_{n\in\N}\subset\R^d\times \R$ such that
the functions $y\mapsto (e^{it_n\sqrt{-\Delta+s^2}}g_{n,1})(y-y_n)$ have a nonzero uniform limit in $\{y\in\R^d:\ab{y}\leq R\}$.
This implies that weak limits of the sequence $\{e^{it_n\sqrt{-\Delta+s^2}}g_n(\cdot-y_n)\}_{n\in\N}$
in $H^r(\R^d)$ are nonzero. Using
the argument given in the proof of \cite{FVV2}*{Theorem 1.1} or an argument similar to that in
\cite{RQ}*{Proposition 8.3}, we see that all the hypotheses of \cite{FVV}*{Proposition 1.1} are satisfied by the
sequence  $\{e^{it_n\sqrt{-\Delta+s^2}}g_{n}(\cdot\,-y_n)\}_{n\in\N}$, which is extremizing. Therefore the latter sequence is
precompact in $H^r(\R^d)$, and (ii) is satisfied.

Let us now suppose that (i) is satisfied. The existence of the subsequence $\{g_{n_k}\}_{k\in\N}$, the sequences
$\{\la_k\}_{k\in\N}$, and $\{(y_k,t_k)\}_{k\in\N}\subset\R^d\times \R$, and the function $v\in\dot
H^r(\R^d)$, $v\neq 0$, satisfying
\eqref{weak-convergence-homogeneous} follows as in \cite{FVV2}*{Proof of Theorem 1.1}. That $\la_k\to 0$ as $k\to\infty$
follows from (i). Indeed,  if there exists a subsequence $\{\la_{k_l}\}_{l\in\N}$ with
$\inf_{l\in\N}\la_{k_l}>0$, then the functions
$h_k(y):=\la_k^{d/2-r}e^{it_{k}\sqrt{-\Delta+s^2}}g_{n_k}(\la_k(y-y_k))$ satisfy

\[\int_{\ab{y}\leq R}\ab{\hat
h_{k_l}(y)}^2(1+\ab{y}^2)^rdy=\int_{\ab{y}\leq R/\la_{k_l}}\ab{\hat g_{n_{k_l}}(y)}^2(1/\la_{k_l}^2+\ab{y}^2)^rdy\to0 \text{ as }
l\to\infty,\]
for every $R<\infty$, which is impossible since $v\neq 0$.
\end{proof}

\section{The Lorentz invariance}\label{section-on-lorentz}

The Lorentz group is defined as the group of invertible linear transformations in $\R^{d+1}$ preserving the bilinear form
$(x,y)\in \R^{d+1}\times \R^{d+1}\mapsto x\cdot J y$, where $J=\operatorname{diag}(-1,\dots,-1,1)$.

Let us denote by $\mathcal L^+$ the subgroup of Lorentz transformations in $\R^{d+1}$ that preserve $\mathbbm H^d_s$. It is known
that $\sigma_s$ is invariant under the action of $\mathcal L^+$ and moreover is, up to multiplication by scalar, the unique
measure on $\mathbbm H^d_s$ invariant under such Lorentz transformations; see
\cite{RS} where the case $d=3$ is considered. The same argument can be adapted to $d\geq 2$. 

For $t\in(-1,1)$, we define the linear map $L^t:\R^{d+1}\to\R^{d+1}$ by 
\[L^t(\xi_1,\dots,\xi_{d},\tau)=\biggl(\frac{\xi_1+t\tau}{\sqrt{1-t^2}},\xi_2\dots,\xi_{d},\frac{\tau+t\xi_1}{\sqrt{1-t^2}}
\biggr).\]
Then $\{L^t\}_{t\in(-1,1)}$ is a one parameter subgroup of Lorentz transformations contained in $\mathcal L^+$.

For $i,j\in\{1,\dots,d\}$, let $P_{i,j}$ be the linear transformation that swaps the $i^{th}$ and $j^{th}$ components of every
vector in $\R^{d+1}$. More precisely, for $(\xi_1,\dots,\xi_d,\tau)\in\R^{d+1}$,
$P_{i,j}(\xi_1,\dots,\xi_d,\tau)=(\xi_{\omega_{i,j}(1)},\dots,\xi_{\omega_{i,j}(d)},\tau)$, where $\omega_{i,j}$ is the
permutation of $\{1,\dots,d\}$ defined by 
\[\omega_{i,j}(k)=\begin{cases}
   j&\text{ if }k=i,\\
   i&\text{ if }k=j,\\
   k&\text{ otherwise}.
  \end{cases}
\]

For every orthogonal matrix $A\in O(d,\R)$, the transformation $(\xi,\tau)\mapsto R_A(\xi,\tau)=(A\xi,\tau)$ belongs to $\mathcal
L^+$.

Composing the transformations $P_{i,j}$ and $L^t$ for suitable $i,j$'s and $t$'s, we easily see that if
$(\xi,\tau)\in\R^{d+1}$ satisfies $\tau>\ab{\xi}$, then there exists $L\in\mathcal L^+$ such that
$L(\xi,\tau)=(0,\sqrt{\tau^2-\xi^2})$. Alternatively, this can be seen using the transformations $R_A$ and $L^t$. We first
find $A\in O(d,\R)$ such that $A\xi=(\ab{\xi},0,\dots,0)$. We take $t=-\ab{\xi}\tau^{-1}$ and note that
$L^t(R_A(\xi,\tau))=L^t(\ab{\xi},0,\dots,0,\tau)=(0,\sqrt{\tau^2-\ab{\xi}^2})$.

For $p\in[1,\infty]$, $L\in \mathcal L^+$, and $f\in L^p(\mathbbm H^d_s)$ we define $L^* f=f\circ L$; here ``$\circ$'' denotes
composition. The invariance of the measure $\sigma_s$ under the action of $\mathcal L^+$ implies that $\norma{f}_{L^p(\mathbbm
H^d_s)}=\norma{L^* f}_{L^p(\mathbbm H^d_s)}$
for all $p\in[1,\infty]$, equality holding for $p=\infty$ since Lorentz transformations are invertible. It is also straightforward
to check that $\norma{T_s(L^* f)}_{L^p(\R^{d+1})}=\norma{T_s f}_{L^p(\R^{d+1})}$ for $p\in[1,\infty]$. Therefore, if
$\{f_n\}_{n\in\N}$ is an extremizing sequence
for \eqref{restriction-hyperboloid-s} and $\{L_n\}_{n\in\N}\subset \mathcal L^+$, then $\{L_n^* f_n\}_{n\in\N}$ is also an
extremizing sequence for \eqref{restriction-hyperboloid-s}.

We use the Lorentz transformations $P_{i,j},\, R_A$ and $L^t$. The invariance of $\sigma_s$ with
respect to these transformations can be seen directly from an examination of the Jacobians in the change of variables formula.

\section{On Foschi's argument}

For ease of notation, let $\psi_s(x)=\sqrt{s^2+x^2}$ for $s,x\in\R$ and set $\psi:=\psi_1$. We also write $\psi_s(y)$ to mean
$\psi_s(\ab{y})$ for $y\in\R^d$. We define the convolution of measures $\mu,\nu$ on $\R^d$ by duality as
\[\int_{\R^d} gd(\mu*\nu)=\int_{(\R^d)^2} g(x+y)d\mu(x)d\nu(y),\]
for all $g\in C_0(\R^d)$. For a measure $\mu$ on $\R^d$ and $n\geq 1$, $\mu^{(*n)}=\mu*\dots*\mu$ denotes the
$n$-fold convolution of $\mu$ with itself.

The measure $\sigma_s$ on $\mathbbm H^d_s$ has the property that $\sigma_s^{(*n)}$ is supported in the
closure of the region $\mathcal P_{d,n}=\{(\xi,\tau):\tau>\sqrt{(ns)^2+\ab{\xi}^2}\}$. For each fixed $(\xi,\tau)\in \mathcal
P_{d,n}$, we define the measure on $(\R^d)^n$
\[\mu_{(\xi,\tau)}=\ddirac{\tau-\psi_s(x_1)-\dots-\psi_s(x_n)\\ \xi-x_1-\dots -x_n}dx_1\dotsm dx_n.\]
Recall that the Dirac delta measure $\delta_0$ on $\R^d\times\R$, is defined by
\[\langle\delta_0,f\rangle=f(0),\;\text{ for all }f\in\mathcal S(\R^d\times\R).\]
The measure $\mu_{(\xi,\tau)}$ is the pullback of $\delta_0$ on $\R^d\times\R$ by 
$\Phi_{(\xi,\tau)}:(\R^d)^n\to\R^d\times\R$ given by
\[\Phi_{(\xi,\tau)}(x_1,\dots,x_n)=(\xi-x_1-\dots -x_n,\tau-\psi_s(x_1)-\dots-\psi_s(x_n)).\]

As discussed in \cite{Fo}, the pullback is well-defined as long as the differential of $\Phi_{(\xi,\tau)}$ is surjective at the 
points where $\Phi_{(\xi,\tau)}$ vanishes. The differential of $\Phi_{(\xi,\tau)}$ is surjective at a point $(x_1,\dots,x_n)$ if
and only if $x_1,\dots,x_n$ are not all equal. Now $\Phi_{(\xi,\tau)}(x,\dots,x)=0$ if and only if $\tau^2=(ns)^2+\ab{\xi}^2$,
that is, at the boundary of $\mathcal P_{d,n}$. Thus, the pullback is well-defined on $\mathcal P_{d,n}$.

For each $(\xi,\tau)\in \mathcal P_{d,n}$, we define the inner product $\langle\cdot,\cdot\rangle_{(\xi,\tau)}$ and norm
$\norma{\cdot}_{(\xi,\tau)}$ associated to $\mu_{(\xi,\tau)}$ as
\begin{align*}
 \langle F,G\rangle_{(\xi,\tau)}&=\int_{(\R^d)^n} F(x_1,\dots,x_n)\overline{G(x_1,\dots,x_n)}\, d\mu_{(\xi,\tau)}(x_1,\dots,x_n),\\
 \norma{F}_{(\xi,\tau)}^2&=\int_{(\R^d)^n} \ab{F(x_1,\dots,x_n)}^2 \, d\mu_{(\xi,\tau)}(x_1,\dots,x_n).
\end{align*}

What connects this inner product with inequality \eqref{restriction-hyperboloid} is the following identity. For $f_1,\dotsc,f_n\in
L^2(\mathbbm H^d_s)$,
\begin{align*}
 f_1\sigma_s*\dotsb*f_n\sigma_s&=\int_{(\R^d)^n}\frac{f_1(x_1)\dotsm
f_n(x_n)}{\psi_s(x_1)\dotsm\psi_s(x_n)}\delta(\xi-x_1-\dots-x_n) \\
 &\qquad\cdot\delta(\tau-\psi_s(x_1)-\dots-\psi_s(x_n))\,dx_1\dotsm dx_n\\
 &=\int_{(\R^d)^n}\frac{f_1(x_1)\dotsm f_n(x_n)}{\psi_s(x_1)\dotsm\psi_s(x_n)} d\mu_{(\xi,\tau)}(x_1,\dotsc, x_n)\\
 &=\langle F,G\rangle_{(\xi,\tau)},
\end{align*}
where 
\[F(x_1,\dotsc,x_n)=\frac{f_1(x_1)\dotsm f_n(x_n)}{\psi_s(x_1)^{1/2}\dotsm\psi_s(x_n)^{1/2}}\]
and
\[G(x_1,\dotsc,x_n)=\frac{1}{\psi_s(x_1)^{1/2}\dotsm\psi_s(x_n)^{1/2}}.\]

\begin{lemma}
 \label{convolution-foschi}
 If $f\in \mathcal S(\R^d)$, then
 \begin{align}
  \label{inequality-for-norm-convolution}
  \norma{(f\sigma_s)^{(*n)}}_{L^2(\R^d)}&\leq \norma{\sigma_s^{(*n)}}_{L^\infty(\R^d)}^{1/2}\norma{f}_{L^2(\mathbbm H^d_s)}^n.
 \end{align}
Moreover, for $f\neq 0$, equality holds in \eqref{inequality-for-norm-convolution} only if
$\sigma_s^{(*n)}(\xi,\tau)=\norma{\sigma_s^{(*n)}}_{L^\infty(\R^d)}$ for a.e. $(\xi,\tau)$ in the support of $(f\sigma_s)^{(*n)}$.
\end{lemma}
\begin{proof}
 Let $g\in \mathcal S(\R^{d+1})$. By definition of the convolution,
 \begin{align}
  \langle g,(f\sigma_s)^{(*n)}\rangle
  &=\int_{(\R^d)^n} g(x_1+\dotsb+x_n,\psi_s(x_1)+\dotsb+\psi_s(x_n))\nonumber\\
  &\qquad\cdot\frac{f(x_1)\dotsm
f(x_n)}{\psi_s(x_1)\dotsm\psi_s(x_n)}dx_1\dots dx_n\nonumber\\
  &=\int_{(\R^d)^n} \frac{g(x_1+\dotsb+x_n,\psi_s(x_1)+\dotsb+\psi_s(x_n))}{\psi_s(x_1)^{1/2}\dotsm\psi_s(x_n)^{1/2}}\nonumber\\
  &\qquad\cdot\frac{f(x_1)\dotsm
f(x_n)}{\psi_s(x_1)^{1/2}\dotsm\psi_s(x_n)^{1/2}}dx_1\dotsm dx_n\nonumber\\
  &\leq \biggl(\int_{(\R^d)^n} \frac{g^2(x_1+\dots+x_n,\psi_s(x_1)+\dots+\psi_s(x_n))}{\psi_s(x_1)\dotsm\psi_s(x_n)}dx_1\dotsm
dx_n\biggr)^{1/2}\\
&\qquad\cdot\biggl(\int_{(\R^d)^n} \frac{f(x_1)^2\dotsm f(x_n)^2}{\psi_s(x_1)\dotsm\psi_s(x_n)}dx_1\dotsm
dx_n\biggr)^{1/2}\nonumber\\
  &=\langle g^2,\sigma_s^{(*n)}\rangle^{1/2}\norma{f}_{L^2(\mathbbm H_s^d)}^n\nonumber\\
  \label{cauchy-schwarz-convolution}
  &\leq \norma{g}_{L^2(\R^d)}\norma{\sigma_s^{(*n)}}_{L^\infty(\R^d)}^{1/2}\norma{f}_{L^2(\mathbbm H^d_s)}^n.
 \end{align}
 Taking the supremum over $g\in L^2(\R^{d+1})$ proves \eqref{inequality-for-norm-convolution}.
 
Now if $\norma{(f\sigma_s)^{(*n)}}_{L^2(\R^d)}=\norma{\sigma_s^{(*n)}}_{L^\infty(\R^d)}^{1/2}\norma{f}_{L^2(\mathbbm H^d_s)}^n$,
then, taking $g=(f\sigma_s)^{(*n)}$, we must have equality in \eqref{cauchy-schwarz-convolution}, i.e.,
\[\langle ((f\sigma_s)^{(*n)})^2,\sigma_s^{(*n)}\rangle=\norma{(f\sigma_s)^{(*n)}}_{L^2(\mathbbm
H^d_s)}^2\norma{\sigma_s^{(*n)}}_{L^\infty(\R^d)}, \]
which occurs if and only if $\sigma_s^{(*n)}(\xi,\tau)=\norma{\sigma_s^{(*n)}}_{L^\infty(\R^d)}$ for a.e. $(\xi,\tau)$ in the
support of $(f\sigma_s)^{(*n)}$.
\end{proof}

From Lemma \ref{convolution-foschi} and \eqref{convolution-form-hyperboloid}, we obtain the following result.

\begin{cor}
\label{general-upper-bound}
 Let $(d,p)$ satisfy \eqref{numerology} and suppose $p=2k$ is an even integer. Then
 \begin{equation}
  \label{upper-bound-best-constant-convolution}
  \norma{T_sf}_{L^p(\R^{d+1})}\leq (2\pi)^{(d+1)/p}\norma{\sigma_s^{(*k)}}_{L^\infty(\R^{d+1})}^{1/p}\norma{f}_{L^2(\mathbbm H^d_s)},
 \end{equation}
and thus 
\begin{equation}
\label{bound-for-H}
 \mathbf{H}_{d,p,s}\leq (2\pi)^{(d+1)/p}\norma{\sigma_s^{(*k)}}_{L^\infty(\R^{d+1})}^{1/p}.
\end{equation}
\end{cor}

In the three cases of pairs $(d,p)$ that interest us in this paper, \eqref{bound-for-H} gives
\begin{align*}
 \mathbf{H}_{2,4,s}&\leq (2\pi)^{3/4}\norma{\sigma_s*\sigma_s}_{L^\infty(\R^{3})}^{1/4},\\
 \mathbf{H}_{2,6,s}&\leq (2\pi)^{1/2}\norma{\sigma_s*\sigma_s*\sigma_s}_{L^\infty(\R^{3})}^{1/6},\text{ and}\\
 \mathbf{H}_{3,4,s}&\leq 2\pi\norma{\sigma_s*\sigma_s}_{L^\infty(\R^{4})}^{1/4}.
\end{align*}

To prove the nonexistence of extremizers, we use the following result.

\begin{cor}
\label{nonexistence-tool}
 Let $(d,p)$ satisfy \eqref{numerology} and suppose $p=2k$ is an even integer. Suppose that $\mathbf{H}_{d,p}=
(2\pi)^{(d+1)/p}\norma{\sigma^{(*k)}}_{L^\infty(\R^{d+1})}^{1/p}$ and that
$\sigma^{(*k)}(\tau,\xi)<\norma{\sigma^{(*k)}}_{L^\infty(\R^{d+1})}$ for a.e. $(\xi,\tau)$ in the support of $\sigma^{(*k)}$. Then
extremizers for inequality \eqref{restriction-hyperboloid} do not exist for the pair $(d,p)$.
\end{cor}
\begin{proof}
 This is direct consequence of the last assertion in Lemma \ref{convolution-foschi}.
\end{proof}

\begin{lemma}
 \label{convolution-foschi-with-measure}
 If $f\in \mathcal S(\R^d)$, then
 \begin{multline}
 \label{equality-convolution-measure}
  \norma{f\sigma_s^{(*n)}}_{L^2(\R^d)}^2\leq \int_{(\R^d)^n} \frac{f^2(x_1)\dotsm f^2(x_n)}{\psi_s(x_1)\dotsm \psi_s(x_n)}\\
\cdot\sigma_s^{(*n)}(x_1+\dotsb +x_n,\psi_s(x_1)+\dotsb+\psi_s(x_n))\, dx_1\dotsm dx_n.
 \end{multline}
\end{lemma}
\begin{proof}
 We prove the Lemma only for the case $n=2$; the proof for the general case is similar and only requires more notation.
Following Foschi's argument, we write 
\begin{align}
 f\sigma_s*f\sigma_s(\xi,\tau)&=\int_{(\R^d)^2} \frac{f(x)f(y)}{\psi_s(x)\psi_s(y)}\delta(\xi-x-y)
 \delta(\tau-\psi_s(x)-\psi_s(y))\,dx\,dy\nonumber\\
 \label{convolution-again}
 &=\int_{(\R^d)^2} \frac{f(x)f(y)}{\psi_s(x)\psi_s(y)} d\mu_{(\tau,\xi)}(x,y).
\end{align}

From the Cauchy-Schwarz inequality, for $(\xi,\tau)\in \mathcal P_{d,2}$,
\begin{equation}
\label{from-cauchy-schwarz-general}
 \ab{f\sigma_s*f\sigma_s(\tau,\xi)}\leq \Norma{\frac{f(x)f(y)}{\psi_s(x)^{1/2}\psi_s(y)^{1/2}}}_{(\tau,\xi)}
 \Norma{\frac{1}{\psi_s(x)^{1/2}\psi_s(y)^{1/2}}}_{(\tau,\xi)}.
\end{equation}
Now
\begin{equation}
\label{three-fold-convolution-general}
 \Norma{\frac{1}{\psi_s(x)^{1/2}\psi_s(y)^{1/2}}}_{(\tau,\xi)}^2=\sigma_s*\sigma_s(\xi,\tau)
\end{equation}
as can be seen from \eqref{convolution-again} by taking $f\equiv 1$. Then
\begin{align*}
 \norma{f\sigma_s*f\sigma_s}_2^2&\leq\int_{\mathcal
P_{d,2}}\Norma{\frac{f(x)f(y)}{\psi_s(x)^{1/2}\psi_s(y)^{1/2}}}_{(\tau,\xi)}^2 \sigma_s*\sigma_s(\xi,\tau) d\tau\,d\xi\\
 &=\int_{\mathcal P_{d,2}}\int_{(\R^d)^2} \frac{f^2(x)f^2(y)}{\psi_s(x)\psi_s(y)}\ddirac{\tau-\psi_s(x)-\psi_s(y)\\
\xi-x-y}\sigma_s*\sigma_s(\xi,\tau)\,dx\,dy\,d\tau\,d\xi\\
 &=\int_{(\R^d)^2} \frac{f^2(x)f^2(y)}{\psi_s(x)\psi_s(y)}\int_{\mathcal P_{d,2}}\ddirac{\tau-\psi_s(x)-\psi_s(y)\\
\xi-x-y}\sigma_s*\sigma_s(\xi,\tau)\,d\tau\,d\xi\,dx\,dy\\
 &=\int_{(\R^d)^2} \frac{f^2(x)f^2(y)}{\psi_s(x)\psi_s(y)}\sigma_s*\sigma_s(x+y,\psi_s(x)+\psi_s(y))\,dx\,dy.
 \qedhere
\end{align*}
\end{proof}

\section{Nonexistence of extremizers}
In this section, we prove the first part of Theorem \ref{main-theorem} related to the best constants
and nonexistence of extremizers for the adjoint Fourier restriction inequality for 
$\mathbbm H^d$. We start with the computation of
the double and triple convolution of $\sigma_s$ with itself.

\begin{lemma}
\label{convolutions-d2}
 Let $d=2,\, s>0$, and let $\sigma_s$ be the measure on $\mathbbm H_s^2$ given in \eqref{measure-sigma-s}. Then for
$(\xi,\tau)\in\R^2\times\R$,
  \begin{align}
 \label{formula-double-convolution-d2p4}
  \sigma_s*\sigma_s(\xi,\tau)&=\frac{2\pi}{\sqrt{\tau^2-\ab{\xi}^2}}\chi_{\{\tau\geq \sqrt{(2s)^2+\ab{\xi}^2}\}},\\
  \label{formula-double-convolution-d2p6}
  \sigma_s*\sigma_s*\sigma_s(\xi,\tau)&=(2\pi)^2\biggl(1-\frac{3s}{\sqrt{\tau^2-\ab{\xi}^2}}\biggr)\chi_{\{\tau\geq
\sqrt{(3s)^2+\ab{\xi}^2}\}}.
 \end{align}
 In particular, $\norma{\sigma_s*\sigma_s}_{L^\infty(\R^3)}=\pi/s$, and
$\sigma_s*\sigma_s(\xi,\tau)<\norma{\sigma_s*\sigma_s}_{L^\infty(\R^3)}$ for all $(\xi,\tau)$ in the interior of the
support of $\sigma_s*\sigma_s$. Also, $\norma{\sigma_s*\sigma_s*\sigma_s}_{L^\infty(\R^3)}=(2\pi)^2$, and for all
$(\xi,\tau)\in\R^{d+1}$, $\sigma_s*\sigma_s*\sigma_s(\xi,\tau)<\norma{\sigma_s*\sigma_s*\sigma_s}_{L^\infty(\R^3)}$.
\end{lemma}

\begin{proof}
   It is easy to compute the convolution
\begin{align*}
 \sigma_s*\sigma_s(0,\tau)&=\int_{\R^2}
\delta(\tau-2\sqrt{s^2+\ab{y}^2})\frac{dy}{s^2+\ab{y}^2}=2\pi\int_0^\infty\delta(\tau-2\sqrt{s^2+r^2})\frac{rdr}{s^2+r^2}.
\end{align*}
Let $u=2\sqrt{s^2+r^2}$. Then
\begin{align*}
 \sigma_s*\sigma_s(0,\tau)&=2\pi\int_{2s}^\infty \delta(\tau-u)\frac{du}{u}=\frac{2\pi}{\tau}\chi_{\{\tau\geq2s\}}.
\end{align*}
By Lorentz invariance, we obtain
\[\sigma_s*\sigma_s(\xi,\tau)=\frac{2\pi}{\sqrt{\tau^2-\ab{\xi}^2}}\chi_{\{\tau\geq \sqrt{(2s)^2+\ab{\xi}^2}\}}.\]
To compute the triple convolution, we use the expression we just obtained for the double convolution, which yields
\begin{align*}
 \sigma_s*\sigma_s*\sigma_s(0,\tau)&=\int_{\R^2} \sigma*\sigma(-y,\tau-\sqrt{s^2+\ab{y}^2})\frac{dy}{\sqrt{s^2+\ab{y}^2}}\\
 &=(2\pi)^2\int_0^\infty \frac{\chi_{\{\tau-\sqrt{s^2+r^2}\geq
\sqrt{(2s)^2+r^2}\}}}{\sqrt{(\tau-\sqrt{s^2+r^2})^2-r^2}}\frac{rdr}{\sqrt{s^2+r^2}}.
\end{align*}
Let $u=\sqrt{s^2+r^2}$. Then
\begin{align*}
 \sigma_s*\sigma_s*\sigma_s(0,\tau)&=(2\pi)^2\chi_{\{\tau\geq 3s\}}\int_s^{\frac{\tau^2-3s^2}{2\tau}} \frac{du}{\sqrt{(\tau-u)^2-(u^2-s^2)}}\\
 &=(2\pi)^2\chi_{\{\tau\geq 3s\}}\int_s^{\frac{\tau^2-3s^2}{2\tau}} \frac{du}{\sqrt{\tau^2-2\tau u+s^2}}\\
 &=(2\pi)^2\biggl(1-\frac{3s}{\tau}\biggr)\chi_{\{\tau\geq 3s\}}.
\end{align*}
By Lorentz invariance,
\[\sigma_s*\sigma_s*\sigma_s(\xi,\tau)=(2\pi)^2\biggl(1-\frac{3s}{\sqrt{\tau^2-\ab{\xi}^2}}\biggr)\chi_{\{\tau\geq\sqrt{(3s)^2+\ab
{ \xi}^2}\}}.\qedhere\]
\end{proof}

\begin{lemma}
\label{triple-convolution-d3}
 Let $d=3$ and $s>0$. Then for $(\xi,\tau)\in\R^3\times\R$,
 \begin{equation}
 \label{formula-double-convolution-d3p6}
  \sigma_s*\sigma_s(\xi,\tau)=2\pi\biggl(1-\frac{4s^2}{\tau^2-\ab{\xi}^2}\biggr)^{1/2}\chi_{\{\tau\geq
\sqrt{(2s)^2+\ab{\xi}^2}\}}.
  \end{equation}
 In particular, $\norma{\sigma_s*\sigma_s}_{L^\infty(\R^4)}=2\pi$, and for all $(\xi,\tau)\in \R^4$,
\[\sigma_s*\sigma_s(\xi,\tau)<\norma{\sigma_s*\sigma_s}_{L^\infty(\R^4)}.\]
\end{lemma}
 \begin{proof}
\begin{align*}
 \sigma_s*\sigma_s(0,\tau)=\int_{\R^3}\delta(\tau-2\sqrt{s^2+\ab{y}^2})\frac{dy}{s^2+\ab{y}^2}=4\pi\int_0^\infty\delta(\tau-2\sqrt{s^2+r^2})\frac{r^2dr}{s^2+r^2}.
\end{align*}
Let $u=2\sqrt{s^2+r^2}$. Then
\begin{align*}
 \sigma_s*\sigma_s(0,\tau)&=2\pi\int_{2s}^\infty\delta(\tau-u)\frac{\sqrt{u^2-4s^2}}{u}du=2\pi\frac{\sqrt{\tau^2-4s^2}}{\tau}\chi_{\{\tau\geq 2s\}}\\
 &=2\pi\biggl(1-\frac{4s^2}{\tau^2}\biggr)^{1/2}\chi_{\{\tau\geq 2s\}}.
\end{align*}
Therefore, by the Lorentz invariance,
\[\sigma_s*\sigma_s(\xi,\tau)=2\pi\biggl(1-\frac{4s^2}{\tau^2-\ab{\xi}^2}\biggr)^{1/2}\chi_{\{\tau\geq
\sqrt{(2s)^2+\ab{\xi}^2}\}}.\qedhere\]
\end{proof}

From Corollary \ref{general-upper-bound} and Lemmas \ref{convolutions-d2} and \ref{triple-convolution-d3}, we obtain the
following result.

\begin{cor}
 \label{upper-bound-best-constants}
 \[\mathbf{H}_{2,4}\leq 2^{3/4}\pi,\quad\mathbf{H}_{2,6}\leq (2\pi)^{5/6},\quad\text{ and }\quad\mathbf{H}_{3,4}\leq
(2\pi)^{5/4}.\]
\end{cor}

To obtain the lower bound for the best constants, we exhibit explicit extremizing sequences.

\begin{lemma}
\label{extremizing-sequence-d2}
 Let $d=2$ and $s>0$. For $a>0$, let $f_a(y)=e^{-a\sqrt{s^2+\ab{y}^2}}$, $y\in\R^2$. Then
 \begin{align}
  \label{limit-extremizing-sequence-d2p4}
  \lim\limits_{a\to\infty}\norma{T_s f_a}_{L^4(\R^3)}\norma{f_a}_{L^2(\mathbbm H^2_s)}^{-1}&=\frac{2^{3/4}\pi}{s^{1/4}},\\
 \label{extremizing-sequence-d2p6}
 \lim\limits_{a\to0^+}\norma{T_sf_a}_{L^6(\R^3)}\norma{f_a}_{L^2(\mathbbm H^2_s)}^{-1}&=(2\pi)^{5/6}.
  \end{align}
\end{lemma}
The proof of this is given in Appendix 2. For the case $d=3$, we have an analogous result.
\begin{lemma}
\label{linear-function-four}
 Let $d=3$ and $s>0$. For $a>0$, let $f_a(y)=e^{-a\sqrt{s^2+\ab{y}^2}}$, $y\in\R^3$. Then
 \[\lim\limits_{a\to 0^+}\norma{T_sf_a}_{L^4(\R^4)}\norma{f_a}_{L^2(\mathbbm H^3_s)}^{-1}=(2\pi)^{5/4}.\]
\end{lemma}
The proof of Lemma \ref{linear-function-four} is in Appendix 3.

Note that Corollary \ref{upper-bound-best-constants} and Lemmas \ref{extremizing-sequence-d2} and
\ref{linear-function-four}
imply that for $(d,p)=(2,4)$, the family $\{f_a/\norma{f_a}_{L^2(\sigma_s)}\}_{a>0}$ is an extremizing family as $a\to\infty$,
while for $(d,p)=(2,3)$, $\{f_a/\norma{f_a}_{L^2(\sigma_s)}\}_{a>0}$ is an extremizing family as $a\to 0^+$, and for
$(d,p)=(3,6)$,
$\{f_a/\norma{f_a}_{L^2(\sigma_s)}\}_{a>0}$ is an extremizing family as $a\to  0^+$. An extremizing family $\{f_a\}_{a>0}$ is
defined as in Definition \ref{definition-extremizer}, where we replace the limit in $n$ by a limit in $a$.

We now give the proof of the part of Theorem \ref{main-theorem} related to the best constants and extremizers for the adjoint
Fourier restriction inequality on $\mathbbm{H}_s^d$; the proof of the second part, related to the two-sheeted
hyperboloid $\bar{\mathbbm{H}}_s^d$, is deferred to Section \ref{two-sheeted}.

\begin{proof}[Proof of the first part of Theorem \ref{main-theorem}]
Combining Corollary \ref{upper-bound-best-constants} and Lemmas \ref{extremizing-sequence-d2} and
\ref{linear-function-four}, we obtain the first part of Theorem \ref{main-theorem}, namely
\[\mathbf{H}_{2,4}= 2^{3/4}\pi,\quad\mathbf{H}_{2,6}= (2\pi)^{5/6},\quad\text{ and }\quad\mathbf{H}_{3,6}= (2\pi)^{5/4}.\]
That extremizers do not exist is a consequence of Corollary \ref{nonexistence-tool} and the last assertions about the infinity
norm of the double and triple convolutions of $\sigma$ with itself, contained in Lemmas \ref{convolutions-d2} and 
\ref{triple-convolution-d3}.
\end{proof}

We now prove the assertion given in the Introduction about extremizers for the truncated operator $T_\rho$ for $d=2$ and $p=4$.
\begin{prop}
\label{result_truncated_hyperboloid}
 Let $(d,p)=(2,4)$ and $s>0$. For any $\rho>0$, the best constant in inequality \eqref{restricted-hyperboloid} equals
$2^{3/4}\pi/s^{1/4}$,
and there are no extremizers for inequality \eqref{restricted-hyperboloid}.
\end{prop}
\begin{proof}
 The nonexistence of extremizers for inequality \eqref{restricted-hyperboloid} follows from the nonexistence of extremizers for
inequality for \eqref{restriction-hyperboloid} once we prove that the best 
 constant for the truncated hyperboloid equals the best constant for the entire hyperboloid, $\mathbf{H}_{2,4,s}$. For this, we
need a lower bound. 
 
 Since the extremizing family $\{f_a/\norma{f_a}_{L^2(\sigma_s)}\}_{a>0}$ given in Lemma \ref{extremizing-sequence-d2}
concentrates at $y=0$
as $a\to\infty$, one easily sees that
\[T_\rho( f_a\chi_{\{\ab{y}\leq \rho\}}/\norma{f_a\chi_{\{\ab{y}\leq \rho\}}}_{L^2(\sigma_s)})\to  2^{3/4}\pi/s^{1/4}\;\text{, as
}\;a\to\infty\]
 for the family $\{f_a\chi_{\{\ab{y}\leq \rho\}}/\norma{f_a\chi_{\{\ab{y}\leq \rho\}}}_{L^2(\sigma_s)}\}_{a> 0}$. This gives the
desired
lower bound.
\end{proof}

\section{On extremizing sequences}

In this section, we obtain some general
properties concerning concentration of extremizing sequences for inequality \eqref{restriction-hyperboloid} for the cases
$(d,p)=(2,4),(2,6)$ and $(3,4)$. 

The Lorentz
invariance of $\sigma_s$ implies that given an extremizing sequence $\{f_n\}_{n\in\N}$ for inequality
\eqref{restriction-hyperboloid-s} and a sequence of Lorentz transformations $\{L_n\}_{n\in \N}$ preserving $\mathbbm H^d_s$,
$\{f_n\circ L_n\}_{n\in\N}$ is also an extremizing sequence. In this section, it is only in the
case $(d,p)=(2,4)$ that the Lorentz group is used explicitly, but an equivalent result can be written without it, as discussed
before the statement of Proposition \ref{spatial-infinity-d2p4-with-ds}.

Consider first the case $d=2$ and $p=6$. From Lemma \ref{extremizing-sequence-d2}, it follows that the family of functions
$\{f_a/\norma{f_a}_2\}_{a>0}$ is an extremizing family as $a\to 0^+$. This particular extremizing family concentrates at
spatial infinity, that is, for every $\eps,R>0$, there exists $a_0>0$ such that for all $0<a<a_0$,
$\norma{f_a/\norma{f_a}_2}_{L^2(B(0,R))}<\eps$, where $B(0,R)=\{y\in\R^2:\ab{y}<R\}$. Next we show that this is the case for every
extremizing sequence.

\begin{prop}
\label{spatial-infinity-three}
 Let $\{f_n\}_{n\in\N}$ be an extremizing sequence for inequality \eqref{restriction-hyperboloid-s} in the case $(d,p)=(2,6)$.
Then
for any $\eps,R>0$, there exists $N\in\N$ such that for $n\geq N$,
 \[\norma{f_n}_{L^2(B(0,R))}<\eps,\]
 that is, the sequence concentrates at spatial infinity.
\end{prop}
\begin{proof}
Let $\eps,R>0$ be given. From the proof of Lemma \ref{convolution-foschi-with-measure} and from Lemma \ref{convolutions-d2},  we
have for the inequality in convolution form
\begin{align*}
 \norma{f_n\sigma_s*&f_n\sigma_s*f_n\sigma_s}_{L^2(\R^3)}^2\\
 &\leq \int_{\mathcal
P_{2,3}}\Norma{\frac{f_n(x)f_n(y)f_n(z)}{\psi_s(x)^{1/2}\psi_s(y)^{1/2}\psi_s(z)^{1/2}}}_{(\tau,\xi)}^2
\sigma_s*\sigma_s*\sigma_s(\tau,\xi) d\tau d\xi \\
 &=(2\pi)^2\int_{\mathcal
P_{2,3}}\Norma{\frac{f_n(x)f_n(y)f_n(z)}{\psi_s(x)^{1/2}\psi_s(y)^{1/2}\psi_s(z)^{1/2}}}_{(\tau,\xi)}^2\biggl(1-\frac{
3s } { \sqrt{\tau^2-\ab{\xi}^2}}\biggr)d\tau d\xi\\
 &=(2\pi)^2\norma{f_n}_{L^2(\sigma_s)}^6-(2\pi)^2\int_{\mathcal
 P_{2,3}}\Norma{\frac{f_n(x)f_n(y)f_n(z)}{\psi_s(x)^{1/2}\psi_s(y)^{1/2}\psi_s(z)^{1/2}}}_{(\tau,\xi)}^2\frac{
3s\,d\tau d\xi}{\sqrt{\tau^2-\ab{\xi}^2}}.
 \end{align*}
 Since $\norma{f_n\sigma_s*f_n\sigma_s*f_n\sigma_s}_{L^2(\R^3)}^2\to (2\pi)^2$ as $n\to\infty$, we obtain 
 \begin{equation}
 \label{extremizing-sequence-to-zero}
  \int_{\mathcal
P_{2,3}}\Norma{\frac{f_n(x)f_n(y)f_n(z)}{\psi_s(x)^{1/2}\psi_s(y)^{1/2}\psi_s(z)^{1/2}}}_{(\tau,\xi)}^2\frac{
d\tau d\xi}{\sqrt{\tau^2-\ab{\xi}^2}}\to 0\,\text{ as }n\to\infty;
 \end{equation}
 and thus there exists $N\in \N$ such that for all $n\geq N$,
 \[\int_{\mathcal
P_{2,3}}\Norma{\frac{f_n(x)f_n(y)f_n(z)}{\psi_s(x)^{1/2}\psi_s(y)^{1/2}\psi_s(z)^{1/2}}}_{(\tau,\xi)}^2\frac{
d\tau d\xi}{\sqrt{\tau^2-\ab{\xi}^2}}< \frac{\eps}{3\sqrt{s^2+R^2}}.\]
 By Lemma \ref{convolution-foschi-with-measure}, 
 the expression in the left hand side can be written as
 \begin{align*}
  &\int_{\mathcal
P_{2,3}}\Norma{\frac{f_n(x)f_n(y)f_n(z)}{\psi_s(x)^{1/2}\psi_s(y)^{1/2}\psi_s(z)^{1/2}}}_{(\tau,\xi)}^2\frac{
d\tau d\xi}{\sqrt{\tau^2-\ab{\xi}^2}}\\
  &=\int_{(\R^2)^3} \frac{f_n^2(x)f_n^2(y)f_n^2(z)}{\psi_s(x)\psi_s(y)\psi_s(z)}\int_{\mathcal
P_{2,3}}\ddirac{\tau-\psi_s(x)-\psi_s(y)-\psi_s(z)\\ \xi-x-y-z}\frac{d\tau d\xi}{\sqrt{\tau^2-\ab{\xi}^2}}dxdydz\\
  &\geq \int_{(\R^2)^3} \frac{f_n^2(x)f_n^2(y)f_n^2(z)}{\psi_s(x)\psi_s(y)\psi_s(z)}\int_{\mathcal
P_{2,3}}\ddirac{\tau-\psi_s(x)-\psi_s(y)-\psi_s(z)\\ \xi-x-y-z}\frac{1}{\tau}d\tau d\xi\,dx\,dy\,dz\\
  &\geq \int_{(B(0,R))^3}
\frac{f_n^2(x)f_n^2(y)f_n^2(z)}{\psi_s(x)\psi_s(y)\psi_s(z)}\frac{dx\,dy\,dz}{\psi_s(x)+\psi_s(y)+\psi_s(z)}.
 \end{align*}
If $x,y,z\in B(0,R)$, then $3s<\psi_s(x)+\psi_s(y)+\psi_s(z)\leq 3\psi_s(R)=3\sqrt{s^2+R^2}$. Therefore, for all $n\geq N$,
\begin{align*}
\frac{\eps}{3\sqrt{s^2+R^2}}&>\int_{\mathcal
P_{2,3}}\Norma{\frac{f_n(x)f_n(y)f_n(z)}{\psi_s(x)^{1/2}\psi_s(y)^{1/2}\psi_s(z)^{1/2}}}_{(\tau,\xi)}^2\frac{
d\tau d\xi}{\sqrt{\tau^2-\ab{\xi}^2}}\\
&\geq \frac{1}{3\sqrt{s^2+R^2}}\norma{f_n}_{L^2(B(0,R))}^6;
\end{align*}
and so, $\sup_{n\geq N}\norma{f_n}_{L^2(B(0,R))}<\eps$ as desired.
\end{proof}

We now turn to the case $d=3$ and $p=4$. Here we can also prove the analog of Proposition \ref{spatial-infinity-three}, namely,
that extremizing sequences must concentrate at spatial infinity.

\begin{prop}
\label{spatial-infinity-four}
 Let $\{f_n\}_{n\in\N}$ be an extremizing sequence for inequality \eqref{restriction-hyperboloid-s} in the case $(d,p)=(3,4)$.
Then
for any $\eps,R>0$, there exists $N\in\N$ such that for all $n\geq N$,
 \[\norma{f_n}_{L^2(B(0,R))}<\eps,\]
 that is, the sequence concentrates at spatial infinity.
\end{prop}
\begin{proof}
 The proof follows the same lines as that of Proposition \ref{spatial-infinity-three}. Using the convolution form of the
inequality, we obtain the analog of equation \eqref{extremizing-sequence-to-zero},
 \[\int_{\mathcal
P_{3,2}}\Norma{\frac{f_n(x)f_n(y)f_n(z)}{\psi_s(x)^{1/2}\psi_s(y)^{1/2}\psi_s(z)^{1/2}}}_{(\tau,\xi)}
^2\biggl(1-\Bigl(1-\frac{4s^2}{\tau^2-\ab{\xi}^2}\Bigr)^{1/2}\biggr)d\tau d\xi\to 0\,\text{ as }n\to\infty.\]
 
 If we use the bound 
 \[1-\biggl(1-\frac{4s^2}{\tau^2-\ab{\xi}^2}\biggr)^{1/2}\geq 1-\biggl(1-\frac{4s^2}{\tau^2}\biggr)^{1/2}\]
 and the fact that $0<\psi_s(x)+\psi_s(y)\leq 2\psi_s(R)$ whenever $\ab{x},\ab{y}\leq R$, we obtain
 \begin{multline*}
  \int_{\mathcal
P_{3,2}}\Norma{\frac{f_n(x)f_n(y)f_n(z)}{\psi_s(x)^{1/2}\psi_s(y)^{1/2}\psi_s(z)^{1/2}}}_{(\tau,\xi)}
^2\biggl(1-\Bigl(1-\frac{4s^2}{\tau^2-\ab{\xi}^2}\Bigr)^{1/2}\biggr)d\tau d\xi\\
  \geq\biggl(1-\Bigl(\frac{R^2}{R^2+s^2}\Bigr)^{1/2}\biggr)\norma{f_n}_{L^2(B(0,R))}^2.
 \end{multline*}
The conclusion follows as in the proof of Proposition \ref{spatial-infinity-three}.
\end{proof}

We now analyze the last case $(d,p)=(2,4)$. Since $\sigma_s*\sigma_s(\xi,\tau)=\norma{\sigma_s*\sigma_s}_{L^\infty(\R^3)}$
whenever $\tau=\sqrt{(2s)^2+\ab{\xi}^2}$, that is, at the boundary of the support of $\sigma_s*\sigma_s$, it is not hard to see
that there are extremizing
sequences that concentrate at any given point in $\mathbbm H^2_s$. For the example of an extremizing sequence given in Lemma
\ref{extremizing-sequence-d2}, 
the concentration occurs at the vertex of the hyperboloid $(\xi,\tau)=(0,s)=:P$. We want to show that all extremizing sequences
concentrate.

Since every point in $\mathbbm H^2_s$ has an extremizing sequence concentrating at it, we can construct an
extremizing 
sequence that concentrates along any given sequence in $\mathbbm H^2_s$ in the sense that given a sequence
$\{y_n\}_{n\in\N}\subset
\mathbbm H^2_s$, there exists an extremizing sequence $\{f_n\}_{n\in\N}\subset L^2(\mathbbm H^2_s)$ with the property that for
every $\eps>0$ and $r>0$, there exists $N\in\N$ such that
\begin{equation}
 \label{concentration-along-sequence}
 \int_{\ab{y-y_n}>r}\ab{f_n(y)}^2\, d\sigma_s(y)\leq \eps
\end{equation}
 for all $n\geq N$. Equivalently, taking a Lorentz transformation $L_n\in \mathcal L^+$ with $L_n^{-1}(y_n)=(0,s)=P$ and using
the Lorentz
invariance of the measure $\sigma_s$, we can write \eqref{concentration-along-sequence} as
\[\int_{L_n^{-1}(\{y:\ab{y}>r\})+P}\ab{L_n^*f_n(y)}^2\, d\sigma_s(y)\leq \eps,\]
where $L_n^*f_n(y)=f_n(L_n y)$. We show next that this is the case for every extremizing sequence.

\begin{prop}
\label{spatial-infinity-d2p4-with-ds}
 Let $\{f_n\}_{n\in\N}$ be an extremizing sequence for inequality \eqref{restriction-hyperboloid-s} in the case $(d,p)=(2,4)$.
There
exists a sequence
$\{L_n\}_{n\in\N}\subset \mathcal{L}^+$ with the property that for all $\eps,r>0$, there exists $N\in\N$ such that
 \begin{equation}
  \label{concentration-ds-ball}
  \int_{\ab{y-P}>r}\ab{L_n^*f_n(y)}^2\,d\sigma_s(y)\leq \eps,
 \end{equation}
  for all
$n\geq N$, where $P=(0,s)$  is the vertex of the hyperboloid $\mathbbm H_s^2$.
\end{prop}

To prove this proposition, we introduce the function $d_s:\R^2\times\R^2\to\R$ given by the formula
\[d_s(x,y)=\frac{1}{2s}((\psi_s(x)+\psi_s(y))^2-\ab{x+y}^2)^{1/2}-1.\]
Elementary properties of $d_s$ are described in the next lemma, whose proof is left to the reader.

\begin{lemma}
\label{properties-of-ds}
\begin{enumerate}
  \item [(i)] For all $x,y\in\R^2$, $d_s(x,y)=d_s(y,x)\geq 0$, and $d_s(x,y)=0$ if and only if $x=y$.
  \item [(ii)] For all $x\in\R^2$, $\lim_{\ab{y}\to\infty}d_s(x,y)=\infty$.
  \item [(iii)] For every $R>0$, there exist $0<C_1(R),\,C_2(R)<\infty$ such that
  \[C_1(R)\ab{x-y}^2\leq d_s(x,y)\leq C_2(R)\ab{x-y},\]
  for all $x,y$ with $\ab{x},\ab{y}\leq R$.
\end{enumerate}
\end{lemma}
Property (ii) implies that for given $y\in\R^2$, the $d_s$-ball of radius $R>0$ and center $y$, $B_{d_s}
(y,R):=\{x\in\R^2:d_s(x,y)\leq R\}$, is a bounded set. Property (iii) relates the $d_s$-ball to the euclidean ball; namely, it
implies that for $y$ with $\ab{y}\leq R$ and $r>0$
\begin{equation}
 \label{inclusion-balls}
 B(y,cr)\subset B_{d_s}(y,r)\subset B(y,c'\sqrt{r}),
\end{equation}
 for some constants $c,c'$ depending only on $R$ and $r$.
\begin{proof}[Proof of Proposition \ref{spatial-infinity-d2p4-with-ds}]
The first task is to find a sequence $\{y_n\}_{n\in\N}\subset \mathbbm H^2_s$ such that an analog of
\eqref{concentration-along-sequence} is satisfied. It is convenient, for notational purposes only, to identify functions from
$\mathbbm
H^2_s$ to $\R$ with functions from $\R^2$ to $\R$, and points in $\mathbbm H^2_s$ with points in $\R^2$. This is done via the
projection
of $\mathbbm H^2_s$ onto $\R^2\times\{0\}$. 

From Lemmas \ref{convolution-foschi-with-measure} and \ref{convolutions-d2}, for the inequality in convolution form, we have
\begin{align*}
 \norma{f_n\sigma_s*f_n\sigma_s}_{L^2(\R^3)}^2&\leq \int_{\mathcal
P_{2,2}}\Norma{\frac{f_n(x)f_n(y)}{\psi_s(x)^{1/2}\psi_s(y)^{1/2}}}_{(\tau,\xi)}^2 \sigma_s*\sigma_s(\tau,\xi) d\tau d\xi \\
 &=\frac{\pi}{s}\int_{\mathcal
P_{2,2}}\Norma{\frac{f_n(x)f_n(y)}{\psi_s(x)^{1/2}\psi_s(y)^{1/2}}}_{(\tau,\xi)}^2\frac{2s}{\sqrt{\tau^2-\ab{\xi}^2}}d\tau
d\xi\\
 &\leq \frac{\pi}{s}\norma{f_n}_{L^2}^4.
 \end{align*}
 Since $\norma{f_n\sigma_s*f_n\sigma_s}_{L^2(\R^3)}^2\to \pi/s$ as $n\to\infty$, we obtain 
 \begin{equation}
 \label{extremizing-sequence-to-zero-four}
  \int_{\mathcal
P_{2,2}}\Norma{\frac{f_n(x)f_n(y)}{\psi_s(x)^{1/2}\psi_s(y)^{1/2}}}_{(\tau,\xi)}^2\frac{2s}{\sqrt{\tau^2-\ab{\xi}^2}}
\,d\tau d\xi\to 1\,\text{ as }n\to\infty.
 \end{equation}
 As in the proof of Lemma \ref{convolution-foschi-with-measure}, the expression on the left hand side can be written as
 \begin{align*}
  \int_{\mathcal
P_{2,2}}\Norma{&\frac{f_n(x)f_n(y)}{\psi_s(x)^{1/2}\psi_s(y)^{1/2}}}_{(\tau,\xi)}^2\frac{2s}{\sqrt{\tau^2-\ab{\xi}^2}}
\,d\tau d\xi\\
  &=\int_{(\R^2)^2} \frac{f_n^2(x)f_n^2(y)}{\psi_s(x)\psi_s(y)}\int_{\mathcal P_{2,2}}\ddirac{\tau-\psi_s(x)-\psi_s(y)\\
\xi-x-y}\frac{2s}{\sqrt{\tau^2-\ab{\xi}^2}}d\tau d\xi\,dx\,dy\\
  &= \int_{(\R^2)^2} \frac{f_n^2(x)f_n^2(y)}{\psi_s(x)\psi_s(y)}\frac{2s}{((\psi_s(x)+\psi_s(y))^2-\ab{x+y}^2)^{1/2}}dx\,dy\\
  &= \int_{(\R^2)^2} \frac{f_n^2(x)f_n^2(y)}{\psi_s(x)\psi_s(y)}K_s(x,y)dx\,dy.
 \end{align*}
Observe that
\[\int_{(\R^2)^2} \frac{f_n^2(x)f_n^2(y)}{\psi_s(x)\psi_s(y)}dx\,dy=\norma{f_n}_{L^2(\mathbbm H^2_s)}^2\to1 \text{ as
}n\to\infty\]
and
\[K_s(x,y):=\frac{2s}{((\psi_s(x)+\psi_s(y))^2-\ab{x+y}^2)^{1/2}}=\frac{1}{d_s(x,y)+1}\leq 1\]
for all $x,y\in\R^2$.
Equation \eqref{extremizing-sequence-to-zero-four} implies that
\[\int_{(\R^2)^2}\frac{f_n^2(x)f_n^2(y)}{\psi_s(x)\psi_s(y)}K_s(x,y) dx\,dy\to 1\,\text{ as }n\to\infty.\]
Let $h_n(y)=f_n(y)^2/\psi_s(y)$, so that $\lim_{n\to\infty}\int_{\R^2} h_n(y)dy=1$. For $\eps>0$,
\begin{align*}
 \int_{(\R^2)^2} h_n(x)h_n(y)K_s(x,y)dx\,dy&=\int_{\substack{d_s(x,y)\leq \eps}} h_n(x)h_n(y)K_s(x,y)dx\,dy\\
 &\qquad+\int_{\substack{d_s(x,y)>\eps}} h_n(x)h_n(y)K_s(x,y)dx\,dy\\
 &\leq \norma{h_n}_{L^1(\R^2)}^2-\biggl(1-\frac{1}{\eps+1}\biggr)\int_{\substack{d_s(x,y)>\eps}} h_n(x)h_n(y)dx\,dy.
\end{align*}
Since the left hand side tends to $1$ as $n\to \infty$, we conclude that 
\[\lim\limits_{n\to\infty}\int_{\substack{d_s(x,y)\leq \eps}} h_n(x)h_n(y)dx\,dy=1.\]
Using Fubini's Theorem, we can write
\begin{align*}
 \int_{\substack{d_s(x,y)\leq \eps}} h_n(x)h_n(y)dx\,dy&=\int_{\R^2}h_n(y)\int_{\substack{d_s(x,y)\leq \eps}}h_n(x)dxdy\\
 &\leq \norma{h}_{L^1(\R^2)}\sup\limits_{y\in\R^2}\int_{d_s(x,y)\leq \eps}h_n(x)dx.
\end{align*}
Then
\begin{equation}
 \label{concentration-in-the-limit}
 \lim\limits_{n\to\infty}\sup\limits_{y\in \R^2}\int_{B_{d_s}(y,\eps)}h_n(x)dx= 1.
\end{equation}
Equation \eqref{concentration-in-the-limit} implies that there exists $N(\eps)\in\N$ such that for all $n\geq N(\eps)$,
\[\sup\limits_{y\in \R^2}\int_{B_{d_s}(y,\eps)}h_n(y)dy\geq 1-\frac{\eps}{2},\]
and hence there exists $\{y_{n}^\eps\}_{n\geq N(\eps)}\subset \R^2$ such that
\[\int_{B_{d_s}(y_n^\eps,\eps)}h_n(y)dy\geq 1-\eps.\]
Applying \eqref{concentration-in-the-limit} in this way, we obtain, for each $\eps>0$, a number $N(\eps)$ and a sequence
$\{y_n^\eps\}_{n\geq N(\eps)}$.

The construction of the sequence $\{y_n\}_{n\in\N}$ is obtained by a diagonal process. We take a strictly decreasing
sequence $\{\eps_k\}_{k\in\N}$ such that $\eps_k\to0$ as $k\to\infty$. This gives sequences $\{N(k)\}_{k\in\N}$ and
$\{y_n^k\}_{n\geq
N(k),k\geq 0}$. We can take the sequence $\{N(k)\}_{k\in\N}$ strictly increasing. For each $n\geq N(1)$, we let
$l(n)=\sup\{k\in\N:N(k)\leq n\}$. Next, define $\{y_n\}_{n\in\N}$ by
\[y_n=\begin{cases} 
       y_n^{l(n)}\text{ if }n\geq N(1),\\
       y_0^{\phantom{l(n)}}\text{ if }n< N(1),
      \end{cases}
\]
where $y_0\in\R^2$ is arbitrary, but fixed.

Now let $\eps,r>0$ be given. Take $k$ such that $\eps_{k}<\min\{\eps,r\}$. For $n\geq N(k)$, $l(n)\geq k$, so $\eps_{l(n)}\leq
\eps_k<\min\{\eps,r\}$
and $\int_{B_{d_s}(y_n^{l(n)},\eps_{l(n)})}h_n(y)dy\geq 1-\eps_{l(n)}$; hence
\[\int_{B_{d_s}(y_n,r)}h_n(y)dy\geq\int_{B_{d_s}(y_n^{l(n)},\eps_{l(n)})}h_n(y)dy\geq 1-\eps_{l(n)}\geq
1-\eps.\]
Therefore, for every $\eps,r>0$, there exists $N\in\N$ such that
\begin{equation}
 \label{conclusion-with-projection}
 \int_{B_{d_s}(y_n,r)}\ab{f_n(y)}^2\frac{dy}{\sqrt{s^2+\ab{y}^2}}\geq1-\eps.
\end{equation}
 for all $n\geq N$.
 
To finish the proof we use the Lorentz invariance. This is better done without identifying $\mathbbm H^2_s$ with
$\R^2$. So now we lift everything to $\mathbbm H^2_s$. Let
$D_s:\{(\xi,\tau)\in\R^2\times\R:\tau>\ab{\xi}\}^2\to\R$
be defined by
\[D_s((\xi_1,\tau_1),(\xi_2,\tau_2))=\frac{1}{2s}((\tau_1+\tau_2)^2-\ab{\xi_1+\xi_2}^2)^{1/2}-1.\]
Observe that for every $L\in\mathcal L^+$, $D_s(L(\xi_1,\tau_1),L(\xi_2,\tau_2))=D_s((\xi_1,\tau_1),(\xi_2,\tau_2))$.

Let $z_n=(y_n,\psi_s(y_n))\in\mathbbm H^2_s$. We can write \eqref{conclusion-with-projection} as
\[\int_{D_s(z,z_n)\leq r}\ab{f_n(z)}^2 d\sigma_s(z)\geq 1-\eps.\]
By the Lorentz invariance of $D_s$ and $\sigma_s$, we have that for $L_n\in\mathcal L^+$ for which $L_n^{-1}(z_n)=(0,s)=P$ and for
every $\eps,r>0$, there exists $N\in\N$ such that
\begin{equation}
 \label{last-step-concentration}
 \int_{D_s(z,P)\leq r}\ab{L_n^*f_n(z)}^2 d\sigma_s(z)\geq 1-\eps\,\text{ for all  }n\geq N.
\end{equation}
 Property (iii) in Lemma \ref{properties-of-ds} and \eqref{last-step-concentration} imply that for
every $\eps,r>0$, there exists $N\in\N$ such that
\[\int_{\ab{z-P}\leq r}\ab{L_n^*f(z)}^2 d\sigma_s(z)\geq 1-\eps\]
for all $n\geq N$.
\end{proof}

\section{The two-sheeted hyperboloid}\label{two-sheeted}
In this section, we consider the two-sheeted hyperboloid 
\[\bar{\mathbbm H}^d_s=\{(y,y')\in\R^d\times \R:y'^2=s^2+\ab{y}^2\}\]
with measure
\[ \bar\sigma_s(y,y')=\delta(y'-\sqrt{s^2+\ab{y}^2})\frac{dydy'}{\sqrt{s^2+\ab{y}^2}}
 +\delta(y'+\sqrt{s^2+\ab{y}^2})\frac{dydy'}{\sqrt{s^2+\ab{y}^2}}
\]
and the adjoint Fourier restriction operator defined by $\bar T_s f=\widehat{f\bar\sigma_s}$, for $f\in \mathcal S(\R^{d+1})$.

$\bar{\mathbbm H}^d_s$ is the union of the two sheets
\[\mathbbm H^{d,\pm}_s=\{(y,y')\in\R^d\times \R:y'=\pm\sqrt{s^2+\ab{y}^2}\}.\]
What in this section we are calling $\mathbbm H^{d,+}_s$ is what before we denoted by $\mathbbm H^d_s$. In the previous section,
we proved that for $\mathbbm H^{d,+}_s$ (and thus also for $\mathbbm H^{d,-}_s$),
extremizers do not exist for the cases $(d,p)=(2,4),(2,6)$ and $(3,4)$. Here, we show that extremizers for $\bar{\mathbbm H}^d_s$
do not exist either and compute the best constants. 

The adjoint Fourier restriction operator on $\mathbbm H^{d,+}_s$ is denoted by $T_s$, and the adjoint Fourier restriction
operator on $\mathbbm H^{d,-}_s$ by $T_s^-$. For $s=1$, we drop the subscript $s$. For $A,B\subseteq \R^d$, we set
$A+B=\{a+b:a\in A,\, b\in B\}$ and $-A=\{-a:a\in A\}$.

\begin{lemma}
\label{double-sum-hyperboloid}
 For $d\geq 1$,
 \begin{align}
 \label{plus-plus}
  \mathbbm H_s^{d,+} +\mathbbm H_s^{d,+} &\subseteq\{(\xi,\tau)\in\R^d\times \R:\tau\geq \sqrt{(2s)^2+\ab{\xi}^2}\},\\
  \label{plus-minus}
  \mathbbm H_s^{d,+} +\mathbbm H_s^{d,-} &\subseteq\{(\xi,\tau)\in\R^d\times \R:\ab{\tau}\leq \ab{\xi}\},\\
  \label{minus-minus}
  \mathbbm H_s^{d,-} +\mathbbm H_s^{d,-} &\subseteq\{(\xi,\tau)\in\R^d\times \R:\tau\leq -\sqrt{(2s)^2+\ab{\xi}^2}\}.
 \end{align}
\end{lemma}
\begin{proof}
To establish \eqref{plus-plus}, let
$\xi=x+y$ and $\tau=\psi_s(x)+\psi_s(y)$. Thus
 \[\tau^2=(\psi_s(x)+\psi_s(y))^2=2s^2+\ab{x}^2+\ab{y}^2+2(s^2+\ab{x}^2)^{1/2}(s^2+\ab{y}^2)^{1/2},\]
 while $\ab{\xi}^2=\ab{x+y}^2=\ab{x}^2+\ab{y}^2+2x\cdot y$. Then \eqref{plus-plus} is equivalent to the inequality
 \begin{equation}
 \label{plus-plus-equivalent}
  (s^2+\ab{x}^2)^{1/2}(s^2+\ab{y}^2)^{1/2}\geq s^2+x\cdot y \,\text{ for all }x,y\in\R^d.
 \end{equation}
 Using $x\cdot y=\ab{x}\ab{y}\cos\te$, where $\te$ is the angle between $x$ and $y$, we see that \eqref{plus-plus-equivalent} is
equivalent
to
 \begin{equation*}
 (s^2+a^2)^{1/2}(s^2+b^2)^{1/2}\geq s^2+ab
 \end{equation*}
for all $a,b,s\geq 0$, which is easily shown to hold by squaring both sides.

We proceed in a similar way for \eqref{plus-minus}. Let $\xi=x+y$ and $\tau=\psi_s(x)-\psi_s(y)$. Then
$\tau^2=2s^2+\ab{x}^2+\ab{y}^2-2(s^2+\ab{x}^2)^{1/2}(s^2+\ab{y}^2)^{1/2}$. As before, \eqref{plus-minus} is equivalent
to the inequality
\[-(s^2+\ab{x}^2)^{1/2}(s^2+\ab{y}^2)^{1/2}\leq -s^2+x\cdot y\]
for all $x,y\in\R^d$, which in turn is equivalent to
\[(s^2+a^2)^{1/2}(s^2+b^2)^{1/2}\geq s^2+ab,\]
 which holds for all real numbers $a,b,s\geq 0$.
 
As for \eqref{minus-minus}, it follows from \eqref{plus-plus} observing that $\mathbbm H_s^{d,-}=-\mathbbm H_s^{d,+}$.
\end{proof}

\begin{lemma}
\label{triple-sum-hyperboloid}
Let $d\geq 1$,
 \begin{align}
 \label{plus-plus-plus}
 \mathbbm H_s^{d,+}+\mathbbm H_s^{d,+}+\mathbbm H_s^{d,+} &\subseteq\{(\xi,\tau)\in\R^d\times \R:\tau\geq \sqrt{(3s)^2+\ab{\xi}^2}\},\\
 \label{minus-minus-minus}
 \mathbbm H_s^{d,-}+\mathbbm H_s^{d,-}+\mathbbm H_s^{d,-}&\subseteq\{(\xi,\tau)\in\R^d\times \R:\tau\leq -\sqrt{(3s)^2+\ab{\xi}^2}\},\\
 \label{plus-plus-minus}
 \mathbbm H_s^{d,+}+\mathbbm H_s^{d,+}+\mathbbm H_s^{d,-}&\subseteq\{(\xi,\tau)\in\R^d\times \R:\tau\geq
-\sqrt{s^2+\ab{\xi}^2}\},\\
 \label{plus-minus-minus}
 \mathbbm H_s^{d,+}+\mathbbm H_s^{d,-}+\mathbbm H_s^{d,-}&\subseteq\{(\xi,\tau)\in\R^d\times \R:\tau\leq
\sqrt{s^2+\ab{\xi}^2}\}.
\end{align}
\end{lemma}
\begin{proof}
 We know from Lemma \ref{double-sum-hyperboloid} that
 \[\mathbbm H_s^{d,+} +\mathbbm H_s^{d,+} \subseteq\{(\xi,\tau):\tau\geq \sqrt{(2s)^2+\ab{\xi}^2}\}.\]
 We start with \eqref{plus-plus-plus}. Setting $\xi=x+y$ and $\tau\geq \psi_{2s}(x)+\psi_s(y)>0$ and squaring the latter
inequality for $\tau$ gives
 \[\tau^2\geq 5s^2+\ab{x}^2+\ab{y}^2+2(4s^2+\ab{x}^2)^{1/2}(s^2+\ab{y}^2)^{1/2}.\]
 Then \eqref{plus-plus-plus} follows from the inequality
 \[(4s^2+\ab{x}^2)^{1/2}(s^2+\ab{y}^2)^{1/2}\geq 2s^2+x\cdot y\text{ for all }x,y\in\R^d,\]
which is equivalent to
 \[(4s^2+a^2)^{1/2}(s^2+b^2)^{1/2}\geq 2s^2+ab,\]
  which is easy to verify for all $a,b,s\geq 0$.
  
 We now establish \eqref{plus-plus-minus}. Let $\xi=x+y$ and $\tau\geq \psi_{2s}(x)-\psi_s(y)$. If $\tau\geq 0$, we are done. So,
we suppose that $0\geq\tau\geq \psi_{2s}(x)-\psi_s(y)$. Then
 \[\tau^2\leq 5s^2+\ab{x}^2+\ab{y}^2-2(4s^2+\ab{x}^2)^{1/2}(s^2+\ab{y}^2)^{1/2},\]
 and \eqref{plus-plus-minus} follows from the inequality
 \[-(4s^2+\ab{x}^2)^{1/2}(s^2+\ab{y}^2)^{1/2}\leq -2s^2+x\cdot y\text{ for all }x,y\in\R^d\]
which is equivalent to
 \[(4s^2+a^2)^{1/2}(s^2+b^2)^{1/2}\geq 2s^2+ab,\]
which holds for all $a,b,s\geq 0$.
 
 Both \eqref{minus-minus-minus} and \eqref{plus-minus-minus} can be proved similarly or obtained from \eqref{plus-plus-plus} and
\eqref{plus-plus-minus} using that $\mathbbm H_s^{d,-}=-\mathbbm H_s^{d,+}$.
\end{proof}

For a function $f\in L^2(\bar{\mathbbm H}^d_s)$, we write $f=f_++f_-$, where $f_+$ is supported on $\mathbbm H^{d,+}_s$ and
$f_-$ is supported on $\mathbbm H^{d,-}_s$. Then
\[\norma{f}_{L^2(\bar{\mathbbm H}^d_s)}^2=\norma{f_+}_{L^2({\mathbbm H}^{d,+}_s)}^2+\norma{f_-}_{L^2({\mathbbm H}^{d,-}_s)}^2.\]

\begin{prop}
\label{two-sheeted-hyperboloid-p4}
 Let $d\in\{2,3\}$ and $f\in L^2(\bar{\mathbbm H}^d_s)$, $f\neq 0$. Then
 \begin{equation}
 \label{upper-bound-two-sheeted-p4}
  \norma{\bar T_sf}_{L^4(\R^{d+1})}^4\norma{f}_{L^2(\bar{\mathbbm H}^d_s)}^{-4}\leq \frac{3}{2}\mathbf{H}_{d,4,s}^4.
 \end{equation}
 If equality holds in \eqref{upper-bound-two-sheeted-p4}, 
 \[\norma{T_sf_+}_{L^4(\R^{d+1})}=\mathbf{H}_{d,4,s}\norma{f_+}_{L^2({\mathbbm H}^{d,+}_s)}\,\text{ and
}\,\norma{T_s^{-}f_-}_{L^4(\R^{d+1})}=\mathbf{H}_{d,4,s}\norma{f_-}_{L^2({\mathbbm H}^{d,-}_s)}.\]
 
 Moreover, if $\{f_n\}_{n\in\N}$ is an extremizing sequence for $\bar T_s$, then $\{f_{n,+}/\norma{f_{n,+}}_2\}_{n\in\N}$ and
$\{f_{n,-}/\norma{f_{n,-}}_2\}_{n\in\N}$ are extremizing sequences for $T_s$ in $\mathbbm H^{d,+}_s$ and $T_s^-$  in $\mathbbm
H^{d,-}_s$, respectively.
\end{prop}

\begin{proof}
 The proof of \eqref{upper-bound-two-sheeted-p4} is analogous to the argument in \cite{Fo}*{pp. 754-755}. We
restrict attention to the case $s=1$, but the other cases follow in
the same way or by the use of scaling.

Observe that 
\begin{align*}
 \norma{\bar Tf}_{L^4}^4&=\norma{Tf_+ + T^{-}f_-}_{L^4}^4=\norma{(Tf_+ + T^{-}f_-)^2}_{L^2}^2\\
 &=\norma{(Tf_+)^2 + (T^{-}f_-)^2+2(Tf_+)(T^-f_-)}_{L^2}^2.
\end{align*}
 Using the fact that product transforms into convolution under the Fourier transform, we see that the Fourier transforms of
$(Tf_+)^2,\,
(T^-f_-)^2$ and $(Tf_+)(T^-f_-)$ are supported on $\mathbbm H^{d,+}+\mathbbm H^{d,+},\,\mathbbm H^{d,-}+\mathbbm H^{d,-}$, and
$\mathbbm H^{d,+}+\mathbbm H^{d,-}$, respectively. By Lemma
\ref{double-sum-hyperboloid}, the pairwise intersections of these three sets have measure zero. Therefore,
 \begin{align}
 \norma{\bar Tf}_{L^4}^4&=\norma{Tf_+}_{L^4}^4 + \norma{T^{-}f_-}_{L^4}^4+4\norma{(Tf_+)(T^-f_-)}_{L^2}^2\\
 \label{first-inequality-double-hyperboloid}
 &\leq \textbf{H}_{d,4}^4(\norma{f_+}_{L^2}^4 + \norma{f_-}_{L^2}^4+4\norma{f_+}_{L^2}^2\norma{f_-}_{L^2}^2)\\
 \label{second-inequality-double-hyperboloid}
 &\leq \frac{3}{2}\textbf{H}_{d,4}^4(\norma{f_+}_{L^2}^2+\norma{f_-}_{L^2}^2)^2\\
 &=\frac{3}{2}\textbf{H}_{d,4}^4\norma{f}_{L^2}^4,
\end{align}
where we have used the sharp inequality (as in \cite{Fo}) 
\begin{equation}
 \label{inequality-reals-one-hyperboloid}
 X^2+Y^2+4XY\leq \frac{3}{2}(X+Y)^2,\quad X,Y\geq 0,
\end{equation}
where equality holds if and only if $X=Y$. Thus,
\begin{equation}
 \label{upper-bound-two-sheeted-p4-again}
 \norma{\bar Tf}_{L^4(\R^{d+1})}^4\norma{f}_{L^2(\bar{\mathbbm H}^d)}^{-4}\leq \frac{3}{2}\mathbf{H}_{d,4}^4.
\end{equation}
For $f\neq 0$, equality holds in \eqref{upper-bound-two-sheeted-p4-again} if and only if it holds in
\eqref{first-inequality-double-hyperboloid} and \eqref{second-inequality-double-hyperboloid}. Equality holds in
\eqref{first-inequality-double-hyperboloid} if and only if $\norma{Tf_+}_{L^4}=\mathbf{H}_{d,4}\norma{f_+}_{L^2({\mathbbm
H}^{d,+})}$, $\norma{T^-f_-}_{L^4}=\mathbf{H}_{d,4}\norma{f_-}_{L^2({\mathbbm H}^{d,-})}$ and $\ab{Tf_+}=\la\ab{T^-f_-}$ a.e. in
$\R^{d+1}$ for some $\la\geq 0$, and in \eqref{second-inequality-double-hyperboloid} if and only if
$\norma{f_+}_2=\norma{f_-}_2$. Note that equality in \eqref{second-inequality-double-hyperboloid} implies that $\la=1$.

Let $\{f_n\}_{n\in\N}$ be an extremizing sequence for $\bar T$, so that
$\lim_{n\to\infty}\norma{\bar Tf_n}_{L^4(\R^d)}=\bar{\mathbf{H}}_{d,4}$ and
$\norma{f_n}_2\leq 1$. For the decomposition $f_n=f_{n,+}+f_{n,-}$, we see that
\[\lim\limits_{n\to\infty}(\norma{f_{n,+}}_{L^2}^4 +
\norma{f_{n,-}}_{L^2}^4+4\norma{f_{n,+}}_{L^2}^2\norma{f_{n,-}}_{L^2}^2)=\frac{3}{2}.\]
This implies that if $\lim_{n\to\infty}\norma{f_{n,+}}_{L^2}$ and $\lim_{n\to\infty}\norma{f_{n,-}}_{L^2}$ exist, then they must
be equal, and thus equal to $1/\sqrt{2}$. Therefore, any subsequence has a convergent subsequence with limit $1/\sqrt{2}$. This
implies the existence of both limits and
\[\lim\limits_{n\to\infty}\norma{f_{n,+}}_{L^2}=\lim\limits_{n\to\infty}\norma{f_{n,-}}_{L^2}=\frac{1}{\sqrt{2}}.\]

If we write 
\[\norma{Tf_{n,+}}_{L^4}=a_n\mathbf{H}_{d,4}\norma{f_{n,+}}_{L^2({\mathbbm H}^{d,+})}\text{ and }
\norma{T^-f_{n,-}}_{L^4}=b_n\mathbf{H}_{d,4}\norma{f_{n,-}}_{L^2({\mathbbm H}^{d,-})},\]
then, as before,
$\lim_{n\to\infty}a_n\norma{f_{n,+}}_2=1/\sqrt{2}$, and so $\lim_{n\to\infty}a_n=1$; similarly,\break
$\lim_{n\to\infty}b_n=1$. Hence, $\{f_{n,+}/\norma{f_{n,+}}_2\}_{n\in\N}$ and $\{f_{n,-}/\norma{f_{n,-}}_2\}_{n\in\N}$ are
extremizing sequences for $T$ and $T^-$ in $\mathbbm H^{d,+}$ and $\mathbbm
H^{d,-}$, respectively.
 \end{proof}
\begin{cor}
 \label{nonexistence-two-sheeted}
 For $d\in\{2,3\}$, $p=4$ and $s>0$, $\bar{\mathbf{H}}_{d,4,s}=(3/2)^{1/4}\mathbf{H}_{d,4,s}$. Moreover, extremizers for the
adjoint
Fourier restriction
inequality for $\bar{\mathbbm H}^d_s$ do not exist.
\end{cor}
\begin{proof}
 The only part missing is the lower bound for the value of the best constant. For this, take
$\{f_{n,+}\}_{n\in\N}$ to be an extremizing sequence for $T_s$, then, identifying a function on $\mathbbm H^{d,\pm}_s$ with a
function from $\R^d$ to $\R$, set $f_{n,-}(y)=\overline{f_{n,+}}(-y)$, (the complex conjugate of $f_{n,+}$ evaluated at $-y$), $y\in\R^d$.
Then $\{f_n\}_{n\in\N}=\{(f_{n,+}+f_{n,-})/\sqrt{2}\}_{n\in\N}$ is an extremizing sequence for $\bar T_s$ in $\bar{\mathbbm
H}^d_s$, since
inequalities
\eqref{first-inequality-double-hyperboloid} and \eqref{second-inequality-double-hyperboloid} become equalities in
the limit $n\to\infty$.
\end{proof}

\begin{prop}
\label{two-sheeted-hyperboloid-p6}
 Let $d\in\{1,2\}$, $s>0$ and $f\in L^2(\bar{\mathbbm H}^d_s)$, $f\neq 0$. Then
 \begin{equation}
 \label{upper-bound-two-sheeted-p6}
  \norma{\bar T_sf}_{L^6(\R^{d+1})}^6\norma{f}_{L^2(\bar{\mathbbm H}^d_s)}^{-6}< \frac{25}{4}\mathbf{H}_{d,6,s}^6.
 \end{equation}
%
%
In particular,
\[ \bar{\mathbf{H}}_{d,6,s}\leq \biggl(\frac{5}{2}\biggr)^{1/3}\mathbf{H}_{d,6,s}. \]
When $d=2$ we have the refinement
\[ \bar{\mathbf{H}}_{2,6,s}\leq \biggl(\frac{5}{8}(4+3\sqrt{2})\biggr)^{1/6}\mathbf{H}_{2,6,s}. \]
\end{prop}
\begin{proof}
The proof follows the same lines as \cite{Fo}*{pp. 758-760}, and Proposition
\ref{two-sheeted-hyperboloid-p4} using Lemma
\ref{triple-sum-hyperboloid}. Since we want to highlight that \eqref{upper-bound-two-sheeted-p6} is a strict inequality and that a refinement is possible we provide the details. Let us take $s=1$ as other values of $s$ follow by scaling. We start by writing $\bar{T}f=Tf_++T^{-}f_-$, so that
\begin{align*}
\norma{\bar{T}f}_{L^6}^6&=\norma{Tf_++T^{-}f_-}_{L^6}^6=\norma{(Tf_++T^{-}f_-)^3}_{L^2}^2\\
&=\norma{(Tf_+)^3+3(Tf_+)^2(T^-f_-)+3(Tf_+)(T^-f_-)^2+(T^-f_-)^3}_{L^2}^2.
\end{align*}
The Fourier transform of the functions $(Tf_+)^3,\,(Tf_+)^2(T^-f_-),\,(Tf_+)(T^-f_-)^2$ and $(T^-f_-)^3$ are supported on $\mathbbm H^{d,+}+\mathbbm H^{d,+}+\mathbbm H^{d,+}$, $\mathbbm H^{d,+}+\mathbbm H^{d,+}+\mathbbm H^{d,-}$, $\mathbbm H^{d,+}+\mathbbm H^{d,-}+\mathbbm H^{d,-}$ and $\mathbbm H^{d,-}+\mathbbm H^{d,-}+\mathbbm H^{d,-}$, respectively. Therefore, using Lemma \ref{triple-sum-hyperboloid} we obtain
\begin{equation}\label{l6-norm-with-crossed-terms}
\begin{split}
\norma{\bar{T}f}_{L^6}^6&=\norma{Tf_+}_{L^6}^6+\norma{T^-f_-}_{L^6}^6+9\norma{(Tf_+)^2(T^-f_-)}_{L^2}^2+9\norma{(Tf_+)(T^-f_-)^2}_{L^2}^2\\
&\qquad+6\langle (Tf_+)^3\,,(Tf_+)^2(T^-f_-)\rangle+6\langle (Tf_+)(T^-f_-)^2\,,(T^-f_-)^3\rangle\\
&\qquad+18\langle(Tf_+)^2(T^-f_-)\,,(Tf_+)(T^-f_-)^2\rangle.
\end{split}
\end{equation}
Using the Cauchy-Schwarz and H\"older's inequalities together with the sharp inequality for $T$ and $T^-$ we obtain
\begin{equation}\label{upper-bound-l6-crossed-terms}
\begin{split}
\norma{\bar{T}f}_{L^6}^6&\leq \mathbf{H}_{d,6}^6(\norma{f_+}_{L^2}^6+\norma{f_-}_{L^2}^6+9\norma{f_+}_{L^2}^4\norma{f_-}_{L^2}^2+9\norma{f_+}_{L^2}^2\norma{f_-}_{L^2}^4\\
&\qquad\qquad+6\norma{f_+}_{L^2}^5\norma{f_-}_{L^2}+6\norma{f_+}_{L^2}\norma{f_-}_{L^2}^5+18\norma{f_+}_{L^2}^3\norma{f_-}_{L^2}^3).
\end{split}
\end{equation}
We now use the numerical inequality from \cite{Fo}*{Lemma 6.6}, namely, for $X,Y\geq 0$
\[ X^6+Y^6+9X^4Y^2+9X^2Y^4+6X^5Y+6XY^5+18X^3Y^3\leq \frac{25}{4}(X^2+Y^2)^3, \]
with equality if and only if $X=Y$. In this way we obtain
\begin{equation}\label{inequality-l6-double-sheeted}
\norma{\bar{T}f}_{L^6}^6\leq \frac{25}{4}\mathbf{H}_{d,6}^6(\norma{f_+}_{L^2}^2+\norma{f_-}_{L^2}^2)^3=\frac{25}{4}\mathbf{H}_{d,6}^6\norma{f}_{L^2}^6.
\end{equation}
From the first part of Theorem \ref{main-theorem} we have the inequalities $\norma{Tf_+}_{L^6}^6\leq \mathbf{H}_{d,6}^6\norma{f_+}_{L^2}^6$ and $\norma{T^-f_-}_{L^6}^6\leq \mathbf{H}_{d,6}^6\norma{f_-}_{L^2}^6$, which are strict whenever $f_+\neq 0$ and $f_-\neq 0$, so that if $f\neq 0$ then \eqref{inequality-l6-double-sheeted} is a strict inequality. More importantly, the inequalities
\begin{align}
\label{first-CS}
\langle (Tf_+)^3\,,(Tf_+)^2(T^-f_-)\rangle&\leq \norma{(Tf_+)^3}_{L^2}\norma{(Tf_+)^2(T^-f_-)}_{L^2}\\
\label{second-CS}
\langle (Tf_+)(T^-f_-)^2\,,(T^-f_-)^3\rangle&\leq \norma{(Tf_+)(T^{-}f_-)^2}_{L^2}\norma{(T^-f_-)^3}_{L^2}\\
\label{third-CS}
\langle(Tf_+)^2(T^-f_-)\,,(Tf_+)(T^-f_-)^2\rangle&\leq \norma{(Tf_+)(T^{-}f_-)^2}_{L^2}\norma{(T^-f_-)^3}_{L^2}
\end{align}
are strict, whenever $f_+,f_-\neq 0$. Indeed, equality in the Cauchy-Schwarz inequality \eqref{first-CS} forces $(Tf_+)^3=\la(Tf_+)^2(T^-f_-)$ for some $\la\in\mathbb C$, $\la\neq 0$, which by the use of the Fourier transform implies that $f_+\sigma^+\ast f_+\sigma^+\ast f_+\sigma^+=\la f_+\sigma^+\ast f_+\sigma^+\ast f_-\sigma^-$, so that the support of $f_+\sigma^+\ast f_+\sigma^+\ast f_-\sigma^-$ is contained in $\mathbbm H^{d,+}+\mathbbm H^{d,+}+\mathbbm H^{d,+}$, which is impossible if $f_+,f_-\neq 0$. A similar argument shows that \eqref{second-CS} and \eqref{third-CS} are strict inequalities when $f_+,f_-\neq 0$.

It was observed by D. Foschi in a related argument that it is possible to sharpen an inequality such as \eqref{first-CS}, \eqref{second-CS} and \eqref{third-CS} which then can be used to obtain a better bound for the best constant $\bar{\mathbf{H}}_{d,6}$. In what follows we adapt the argument to the hyperboloid in the case $d=2$. 

Let us write $\widetilde{f}_-(y)=\overline{f_-(-y)}$, where the overline denotes complex conjugation. Then
\begin{align*}
\langle (Tf_+)^3\,,(Tf_+)^2(T^-f_-)\rangle&=\langle (Tf_+)^3\overline{(T^-f_-)}\,,(Tf_+)^2\rangle
=\langle (Tf_+)^3(T\widetilde{f}_-)\,,(Tf_+)^2\rangle\\
&=(2\pi)^3\langle f_+\sigma\ast f_+\sigma\ast f_+\sigma\ast \widetilde{f}_-\sigma\,,\,f_+\sigma\ast f_+\sigma\rangle\\
&\leq (2\pi)^3\norma{\sigma^{(\ast 4)}\cdot \sigma^{(\ast 2)}}_{L^\infty}^{1/2}\norma{f_+}_{L^2}^5\norma{f_-}_{L^2},
\end{align*}
where in the last line we used an argument as in Lemma \ref{convolution-foschi}. From Lemma \ref{convolutions-d2} we know
\[ \sigma\ast\sigma(\xi,\tau)=\frac{2\pi}{\sqrt{\tau^2-\ab{\xi}^2}}\chi_{\{\tau\geq  \sqrt{2^2+\ab{\xi}^2}\}}, \]
while the fourth convolution can be calculated in a similar way
\[ \sigma^{(\ast 4)}(\xi,\tau)=4\pi^3\frac{(\sqrt{\tau^2-\ab{\xi}^2}-4)^2}{\sqrt{\tau^2-\ab{\xi}^2}}\chi_{\{\tau\geq  \sqrt{4^2+\ab{\xi}^2}\}}. \]
Then
\[ \norma{\sigma^{(\ast 4)}\cdot\sigma^{(\ast 2)}}_{L^\infty(\R^3)}=\Norma{8\pi^4\frac{(\sqrt{\tau^2-\ab{\xi}^2}-4)^2}{\tau^2-\ab{\xi}^2}\chi_{\{\sqrt{\tau^2-\ab{\xi}^2}\geq 4\}}}_{L^\infty(\R^3)}=8\pi^4. \]
We obtain the inequality
\begin{equation}\label{improved-CS}
\langle (Tf_+)^3,(Tf_+)^2(T^-f_-)\rangle\leq 16\sqrt{2}\pi^5\norma{f_+}_{L^2}^5\norma{f_-}_{L^2},
\end{equation}
and a similar method gives improved inequalities for \eqref{second-CS} and \eqref{third-CS} with the same constant on the right hand side. Note that $16\sqrt{2}\pi^5<\mathbf{H}_{2,6}^6=(2\pi)^5$, so there is an improvement over using the Cauchy-Schwarz and H\"older's inequality together with the sharp bound for $T$ and $T^-$. Using \eqref{l6-norm-with-crossed-terms} and \eqref{improved-CS} we can obtain the analog of \eqref{upper-bound-l6-crossed-terms},
\begin{equation}\label{improved-upper-bound-l6-crossed-terms}
\begin{split}
\norma{\bar{T}f}_{L^6}^6&\leq \mathbf{H}_{d,6}^6(\norma{f_+}_{L^2}^6+\norma{f_-}_{L^2}^6+9\norma{f_+}_{L^2}^4\norma{f_-}_{L^2}^2+9\norma{f_+}_{L^2}^2\norma{f_-}_{L^2}^4\\
&\qquad\qquad+3\sqrt{2}\norma{f_+}_{L^2}^5\norma{f_-}_{L^2}+3\sqrt{2}\norma{f_+}_{L^2}\norma{f_-}_{L^2}^5+9\sqrt{2}\norma{f_+}_{L^2}^3\norma{f_-}_{L^2}^3).
\end{split}
\end{equation}
There is the sharp numerical bound 
\[ X^6+Y^6+9X^4Y^2+9X^2Y^4+3\sqrt{2}X^5Y+3\sqrt{2}XY^5+9\sqrt{2}X^3Y^3\leq \frac{5}{8}(4+3\sqrt{2})(X^2+Y^2)^3, \]
for all $X,Y\geq 0$, with equality if and only if $X=Y$. It implies
\[ \norma{\bar{T}f}_{L^6}^6\leq \frac{5}{8}(4+3\sqrt{2})\mathbf{H}_{d,6}^6\norma{f}_{L^2}^6, \]
which is the desired improvement over \eqref{inequality-l6-double-sheeted}.
\end{proof}


Proposition \ref{two-sheeted-hyperboloid-p4} gives the proof of the second part
of Theorem \ref{main-theorem} while Proposition \ref{two-sheeted-hyperboloid-p6} explains the comment in Remark \ref{new-remark}.

\section{Scaling}
Here, we record the scaling for the family of operators $\{T_s\}_{s>0}$. Recall from the Introduction that for $s>0$, $\mathbbm
H^d_s:=\{(y,\sqrt{s^2+\ab{y}^2}):y\in\R^d\}$ is equipped with the measure
$\sigma_s(y,y')=\delta(y'-\sqrt{s^2+\ab{y}^2})dydy'/\sqrt{s^2+\ab{y}^2}$. The operator $T_s$ is defined on $\mathcal S(\R^d)$ by
 \[T_sf(x,t)=\widehat{f\sigma_s}(-x,-t)=\int_{\R^d} e^{ix\cdot y}e^{it\sqrt{s^2+\ab{y}^2}}f(y)\frac{dy}{\sqrt{s^2+\ab{y}^2}}.\]

We want to show that $\mathbf{H}_{d,p,s}$ defined in \eqref{hdps} satisfies \eqref{scaling}. With the change of variables
$v=sy$ in \eqref{fourier-extension-operator-hyperboloid}, we have
\begin{align*}
 Tf(x,t)&=\int_{\R^d} e^{ix\cdot y}e^{it\sqrt{1+\ab{y}^2}}f(y)\frac{dy}{\sqrt{1+\ab{y}^2}}\\
 &=\int_{\R^d} e^{is^{-1}x\cdot y}e^{it\sqrt{1+s^{-2}\ab{y}^2}}f(s^{-1}y)\frac{s^{-d}dy}{\sqrt{1+s^{-2}\ab{y}^2}}\\
 &=s^{-d+3/2}\int_{\R^d} e^{is^{-1}x\cdot y}e^{is^{-1}t\sqrt{s^2+\ab{y}^2}}s^{-1/2}f(s^{-1}y)\frac{dy}{\sqrt{s^2+\ab{y}^2}}
\end{align*}
from which it follows that $s^{d-3/2}Tf(sx,st)=T_s(s^{-1/2}f(s^{-1}\cdot))(x,t)$ and
\[s^{d-3/2-(d+1)/p}\norma{Tf}_{L^p(\R^{d+1})}=\norma{T_ss^{-1/2}f(s^{-1}\cdot)}_{L^p(\R^{d+1})}.\]
On the other hand,
\begin{align*}
 \int_{\R^d} \ab{f(y)}^2\frac{dy}{\sqrt{1+\ab{y}^2}}&=\int_{\R^d} \ab{f(s^{-1}y)}^2 \frac{s^{-d}dy}{\sqrt{1+s^{-2}\ab{y}^2}}\\
 &=s^{-d+2}\int_{\R^d} \ab{s^{-1/2}f(s^{-1}y)}^2\frac{dy}{\sqrt{s^2+\ab{y}^2}},
\end{align*}
that is, $\norma{f}_{L^2(\sigma)}=s^{-(d-2)/2}\norma{s^{-1/2}f(s^{-1}\cdot)}_{L^2(\sigma_s)}$. Thus
\begin{align*}
 s^{(d-1)/2-(d+1)/p}\norma{Tf}_{L^p(\R^{d+1})}\norma{f}_{L^2(\sigma)}^{-1}&=\norma{T_ss^{-1/2}f(s^{-1}\cdot)}_{L^p(\R^{d+1})}
\norma {s^{-1/2}f(s^{-1}\cdot)}_{L^2(\sigma_s)}^{-1},
\end{align*}
and it follows that for all $s>0$,
\begin{align}
 \label{equality-of-constants}
 \mathbf{H}_{d,p,s}&=s^{(d-1)/2-(d+1)/p}\mathbf{H}_{d,p}.
\end{align}

\section{Some explicit calculations for the case \texorpdfstring{$d=2$}{d=2}}

The exponential integral function $\Ei(x)$, $x\neq 0$, is defined by
\begin{equation}
 \label{exponential-integral}
 \Ei(x)=-\int_{-x}^\infty\frac{e^{-t}}{t}dt=\int_{-\infty}^x \frac{e^t}{t}dt
\end{equation}
where the principal value is taken for $x>0$. 
\begin{lemma}
 \label{explicit-computation-functional-linear}
 Let $a>0$ and $f_a(y)=e^{-a\sqrt{s^2+\ab{y}^2}}$, $y\in\R^2$. Then 
 \begin{align}
  \label{expression-functional-linear}
  \norma{T_sf_a}_{L^6(\R^3)}^6\norma{f_a}_{L^2(\sigma_s)}^{-6}&=(2\pi)^5(1-6as-36a^2s^2e^{6as} \Ei(-6as)),\text{ and}\\
  \label{expression-functional-linear-d2p4}
  \norma{T_sf_a}_{L^4(\R^3)}^4\norma{f_a}_{L^2(\sigma_s)}^{-4}&=\frac{2^3\pi^4}{s}(-4ase^{4as}\Ei(-4as)).
 \end{align}
\end{lemma}

\begin{proof}
We first compute the $L^2(\sigma_s)$-norm of $f_a$,
\begin{align*}
 \norma{f_a}_{L^2(\sigma_s)}^2&=\int_{\R^2} e^{-2a\sqrt{s^2+\ab{y}^2}}\frac{dy}{\sqrt{s^2+\ab{y}^2}}=2\pi \int_0^\infty
e^{-2a\sqrt{s^2+r^2}}\frac{r}{\sqrt{s^2+r^2}}dr\\
 &=2\pi\int_s^\infty e^{-2a r}dr=\frac{\pi}{a}e^{-2as}.
\end{align*}
The formulas in \eqref{expression-functional-linear} and \eqref{expression-functional-linear-d2p4} are easier to compute in
their equivalent convolution form. Let $g_a(\xi,\tau)=e^{-a\tau}$ and observe that
$f_a\sigma_s*f_a\sigma_s=g_a\sigma_s*g_a\sigma_s$ and $f_a\sigma_s*f_a\sigma_s*f_a\sigma_s=g_a\sigma_s*g_a\sigma_s*g_a\sigma_s$.
Then, because $g_a$ is the exponential of a linear function,
$g_a\sigma_s*g_a\sigma_s(\xi,\tau)=g_a(\xi,\tau)\,\sigma_s*\sigma_s(\xi,\tau)$ and
$g_a\sigma_s*g_a\sigma_s*g_a\sigma_s(\xi,\tau)=g_a(\xi,\tau)\,\sigma_s*\sigma_s*\sigma_s(\xi,\tau)$. Therefore,
\begin{align*}
 \norma{f_a&\sigma_s*f_a\sigma_s*f_a\sigma_s}_{L^2(\R^3)}^2\\
 &=\int_{\R\times\R^2}
e^{-2a\tau}(2\pi)^4\biggl(1-\frac{3s}{\sqrt{\tau^2-\ab{\xi}^2}}\biggr)^2\chi_{\{\tau\geq \sqrt{(3s)^2+\ab{\xi}^2}\}}d\tau d\xi\\
 &=(2\pi)^5\int_{3s}^\infty \int_0^{\sqrt{\tau^2-(3s)^2}} e^{-2a\tau}\biggr(1-\frac{3s}{\sqrt{\tau^2-r^2}}\biggl)^2 rdr d\tau\\
 &=(2\pi)^5\int_{3s}^\infty\int_0^{\sqrt{\tau^2-(3s)^2}}
 e^{-2a\tau}\biggl(r+(3s)^2\frac{r}{\tau^2-r^2}-6s\frac{r}{\sqrt{t^2-r^2}}\biggr)drd\tau\\
 &=(2\pi)^5\int_{3s}^\infty e^{-2a\tau}(\tfrac{1}{2}(\tau^2-(3s)^2)+(3s)^2(\log \tau-\log(3s))-6s(\tau-3s))d\tau\\
 &=(2\pi)^5\biggl(\frac{1}{2}\int_{3s}^\infty e^{-2a\tau}\tau^2d\tau+(3s)^2\int_{3s}^\infty e^{-2a\tau}\log \tau d\tau\\
 &\qquad\qquad\qquad-6se^{-6as}\int_{0}^\infty e^{-2a\tau}\tau d\tau -(\tfrac{9}{2}s^2+(3s)^2\log(3s))\int_{3s}^\infty
e^{-2a\tau}d\tau\biggr)\\
 &=(2\pi)^5\biggl(\frac{e^{-6as}(1+6as(1+3as))}{8a^3}+(3s)^2\frac{e^{-6as}\log(3s)-\Ei(-6as)}{2a}-\frac{6se^{-6as}}{4a^2}\\
 &\qquad\qquad\qquad-\biggl(\frac{9}{2}s^2+(3s)^2\log(3s)\biggr)\frac{e^{-6as}}{2a}\biggr).
\end{align*}
Rearranging terms, we have
\[\norma{f_a\sigma_s*f_a\sigma_s*f_a\sigma_s}_{L^2(\R^3)}^2=(2\pi)^5e^{-6as}\biggl(\frac{1}{8a^3}-(3s)^2\frac{\Ei(-6as)e^{6as}}{2a
}-\frac{6s}{8a^2}\biggr).\]
Thus
\begin{align*}
 \norma{f_a\sigma_s*f_a\sigma_s*f_a\sigma_s}_{L^2(\R^3)}^2\norma{f_a}_{L^2(\sigma_s)}^{-6}&=(2\pi)^5\pi^{-3}a^3\biggl(\frac{1}{
8a^3
} -(3s)^2\frac { \Ei(-6as)e^{6as}}{2a}-\frac{6s}{8a^2}\biggr)\\
 &=(2\pi)^2(1-6as-36a^2s^2e^{6as}\Ei(-6as)).
\end{align*}
For the case of $L^4(\R^3)$,
\begin{align*}
 \norma{f_a\sigma_s*f_a\sigma_s}_{L^2(\R^3)}^2&=\int_{\R\times\R^2} e^{-2a\tau}\frac{(2\pi)^2}{\tau^2-\ab{\xi}^2}\chi_{\{\tau\geq  \sqrt{(2s)^2+\ab{\xi}^2}\}} d\tau d\xi\\
 &=(2\pi)^3\int_{2s}^\infty\int_0^{\sqrt{\tau^2-(2s)^2}}e^{-2a\tau}\frac{r}{\tau^2-r^2}dr d\tau\\
 &=(2\pi)^3\biggl(\frac{e^{-4as}\log(2s)-\Ei(-4as)}{2a}-\log(2s)\frac{e^{-4as}}{2a}\biggr)\\
 &=-(2\pi)^3\frac{\Ei(-4as)}{2a}.
\end{align*}
Thus
\begin{align*}
 \norma{f_a\sigma_s*f_a\sigma_s}_{L^2(\R^3)}^2\norma{f_a}_{L^2(\sigma_s)}^{-4}&=-(2\pi)^3a\frac{\Ei(-4as)}{2\pi^2}e^{4as}\\
 &=\frac{\pi}{s}(-4ase^{4as}\Ei(-4as)).\qedhere
\end{align*}
\end{proof}

\begin{proof}[Proof of Lemma \ref{extremizing-sequence-d2}]
 Using the expressions in Lemma \ref{explicit-computation-functional-linear}, we obtain
 \[\lim\limits_{a\to 0^+}\norma{T_sf_a}_{L^6(\R^3)}^6\norma{f_a}_{L^2(\sigma_s)}^{-6}=\lim\limits_{a\to
0^+}(2\pi)^{5}(1-6as-36a^2s^2e^{6as} \Ei(-6as))^{}=(2\pi)^{5}\]
 and
 \[\lim\limits_{a\to
\infty}\norma{T_sf_a}_{L^4(\R^3)}^4\norma{f_a}_{L^2(\sigma_s)}^{-4}=\lim\limits_{a\to\infty}\frac{2^3\pi^4}{s}(-4ase^{4as}
\Ei(-4as))=\frac{2^3\pi^4}{s}.\qedhere\]
\end{proof}

\begin{remark}
\mbox{}
\begin{enumerate}
 \item[1.] It is not hard to see that the function $a\mapsto 1-a+a^2e^{a} \Ei(-a)$ is strictly decreasing for $a\in
[0,\infty)$ and tends to $0$ as $a\to\infty$ and to $1$ as $a\to 0^+$. Thus
$\norma{T_sf_a}_{L^6(\R^3)}^6\norma{f_a}_{L^2(\sigma_s)}^{-6}$ is a strictly decreasing function of $a$, for each fixed $s$.
 \item[2.] The function $a\mapsto -ae^{a}\Ei(-a)$ is strictly increasing for $a\in[0,\infty)$ and tends to $0$ as $a\to 0^+$ and
to
$1$ as $a\to\infty$. Thus $\norma{T_sf_a}_{L^4(\R^3)}^4\norma{f_a}_{L^2(\sigma_s)}^{-4}$ is a strictly increasing function of $a$,
for
each fixed $s$.
\end{enumerate}
\end{remark}

\section{Some explicit calculations for the case \texorpdfstring{$d=3$}{d=3}}

\begin{proof}[Proof of Lemma \ref{linear-function-four}]
For the $L^2(\sigma_s)$-norm of $f_a$, we have
\begin{align*}
 \norma{f_a}_{L^2(\sigma_s)}^2&=\int_{\R^3} e^{-2a\sqrt{s^2+\ab{y}^2}}\frac{dy}{\sqrt{s^2+\ab{y}^2}}=4\pi\int_0^\infty
e^{-2a\sqrt{s^2+r^2}}\frac{r^2\,dr}{\sqrt{s^2+r^2}}\\
 &=4\pi\int_s^\infty e^{-2au}\sqrt{u^2-s^2}du=\frac{4\pi}{a^2}\int_{as}^\infty e^{-2x}\sqrt{x^2-(as)^2}dx.
\end{align*}
Then 
\[\lim\limits_{a\to 0^+}\frac{a^2}{\pi}\norma{f_a}_{L^2(\sigma_s)}^2=1.\]
Using the convolution form of the inequality, our goal is to show
\[\lim\limits_{a\to 0^+}a^4\norma{f_a\sigma_s* f_a\sigma_s}_{L^2(\R^4)}^2=2\pi^3.\]

As in the proof of Lemma \ref{explicit-computation-functional-linear},
 \begin{align*}
  \norma{f_a\sigma_s* f_a\sigma_s}_{L^2(\R^4)}^2&
  =\int_{\R\times\R^3} e^{-2a\tau}
(2\pi)^2\biggl(1-\frac{4s^2}{\tau^2-\ab{\xi}^2}\biggr)\chi_{\{\tau\geq \sqrt{(2s)^2+\ab{\xi}^2}\}} d\tau d\xi\\
  &=(2\pi)^2 4\pi\int_{2s}^\infty\int_0^{\sqrt{\tau^2-(2s)^2}}e^{-2a\tau}\biggl(1-\frac{4s^2}{\tau^2-r^2}\biggr)r^2dr d\tau\\
  &=16\pi^3 \int_{2s}^\infty
e^{-2a\tau}\biggr(\frac{1}{3}(\tau^2-(2s)^2)^{\frac{3}{2}}+4s^2((\tau^2-(2s)^2)^{1/2}\\
 &\qquad-\tau\log\biggl(\frac{\tau+\sqrt{\tau^2-(2s)^2}}{2s}\biggr)\biggl) d\tau\\
  &=\frac{16\pi^3}{a} \int_{2as}^\infty
e^{-2\tau}\biggr(\frac{1}{3a^3}(\tau^2-(2as)^2)^{\frac{3}{2}}+\frac{4s^2}{a}((\tau^2-(2as)^2)^{1/2}\\
&\qquad-\frac{\tau}{a}\log\biggl(\frac{\tau+\sqrt{\tau^2-(2as)^2}}{2as}\biggr)\biggl) d\tau.
 \end{align*}
Multiplying by $a^4$ and taking the limit as $a\to 0^+$ gives
\begin{align*}
 \lim\limits_{a\to 0^+}a^4\norma{f_a\sigma_s* f_a\sigma_s}_{L^2(\R^4)}^2&=\frac{16\pi^3}{3}\int_0^\infty
e^{-2\tau}\tau^3 d\tau=2\pi^3.\qedhere
\end{align*}
\end{proof}

{\bf Acknowledgments.}
The author thanks his dissertation advisor, Michael Christ, for many helpful comments and suggestions.

\end{document}